\documentclass[11pt,reqno]{amsart}

\usepackage{bm}
\usepackage{amsmath,amssymb,amsfonts,amsopn}
\usepackage{graphicx}
\usepackage{subfigure}
\usepackage{listings}
\usepackage{multirow}
\usepackage{enumitem}
\usepackage{algorithm}
\usepackage{algorithmic}
\usepackage{booktabs}
\usepackage{appendix}
\usepackage{comment}
\usepackage{xcolor}
\usepackage{hyperref}
\usepackage[capitalize,nameinlink]{cleveref}

\ifpdf
  \DeclareGraphicsExtensions{.eps,.pdf,.png,.jpg}
\else
  \DeclareGraphicsExtensions{.eps}
\fi

\newtheorem{theorem}{Theorem}[section]
\newtheorem{lemma}[theorem]{Lemma}
\newtheorem{proposition}[theorem]{Proposition}
\newtheorem{corollary}[theorem]{Corollary}
\theoremstyle{definition}

\newtheorem{example}[theorem]{Example}

\theoremstyle{remark}
\newtheorem{remark}[theorem]{Remark}

\crefname{hypothesis}{Hypothesis}{Hypotheses}
\newcommand{\ignore}[1]{}

\title[IMEX-SDIRK2 mr-ccSAV scheme for Navier--Stokes]{An Efficient IMEX-SDIRK2 mr-ccSAV Scheme for the Forced Navier--Stokes Equations with Uniform-in-Time Enstrophy Bounds}

 \author{Honglin Liao}
 \address{School of Mathematics, Nanjing University of Aerospace and Astraunautics, Nanjing 211106, China}
 \email{liaoh@nuaa.edu.cn}
 \author{Haifeng Wang}
 \address{School of Mathematical Sciences, Eastern Institute of Technology, Ningbo, Zhejiang 315200, China}
 \email{hfwang@eitech.edu.cn}
 \author{Xiaoming Wang}
 \address{School of Mathematical Sciences, Eastern Institute of Technology, Ningbo, Zhejiang 315200, China.}
 \email{wxm@eitech.edu.cn}

 \thanks{The authors are listed in alphabetical order and share first
authorship. Author 3 is the corresponding author.}

\makeatletter
\@namedef{subjclassname@2020}{\textup{2020} Mathematics Subject Classification}
\makeatother

\hypersetup{
  hidelinks,
  pdftitle={An Efficient and Uniform-in-time Bounded IMEX-SDIRK2 Mean-Reverting Concurrent-Correction SAV Scheme for the Navier--Stokes Equations},
  pdfauthor={Honglin Liao, Haifeng Wang, and Xiaoming Wang}
}

\begin{document}

\begin{abstract}
We propose and analyze an IMEX-SDIRK2 mean-reverting concurrent-correction scalar auxiliary variable (mr-ccSAV) method for the forced two-dimensional periodic Navier--Stokes equations in vorticity form. The viscous term is treated by Alexander's SDIRK2 method and advection explicitly. Each stage requires two elliptic solves with the same shifted Laplacian and the solution of either a cubic or a linear scalar algebraic equation. For initial vorticity in $\dot L^s(\Omega)$, $s>2$, a stage solution exists for every positive time step; uniqueness is established separately under an explicit small-step condition. The principal result is a direct, unconditional uniform-in-time enstrophy bound for arbitrary positive time steps. For persistently bounded forcing, this estimate is absorbing: the influence of the initial data decays, and the forcing contribution does not accumulate in time. Under additional regularity, uniformly bounded step sizes, and bounded neighboring step ratios, we also establish uniform-in-time $H^1$ and $H^2$ vorticity bounds without a small-step condition. For smooth solutions, the method converges optimally at second order. Numerical experiments confirm its accuracy, long-time robustness, and effectiveness of a companion embedded time-step selector.
\end{abstract}

\maketitle

\section{Introduction}
The incompressible Navier--Stokes equations (NSE) are a fundamental model for viscous incompressible flows and a canonical testbed for long-time numerical integration. We consider the forced two-dimensional periodic problem in vorticity--streamfunction form
\begin{equation}\label{eq:intro-nse}
\omega_t-\nu\Delta\omega+B(\omega)=f,\qquad
B(\omega)=\nabla^\perp\psi\cdot\nabla\omega,
\qquad -\Delta\psi=\omega,
\end{equation}
on $\Omega=(0,2\pi)^2$, with zero spatial mean. This formulation eliminates the pressure, automatically enforces incompressibility, and retains the dissipative structure of the velocity formulation. In particular,
\begin{equation}\label{eq}
\langle B(\omega),\omega\rangle=0,
\end{equation}
which is central to the enstrophy balance and the long-time dynamics of the two-dimensional NSE.

Long-time boundedness is stronger than finite-time stability. In particular, an estimate of the form
$
E^N\le E^0+C\sum_{n=0}^{N-1}\tau_n|f_n|^2
$
may be unconditional in the time steps but generally grows with the terminal time under persistent bounded forcing. Here, a uniform-in-time estimate means a bound independent of $N$ and $t_N$, whereas an absorbing estimate additionally requires the influence of the initial data to diminish as $t_N\to\infty$. Constructing such estimates is a challenge when nonlinear advection is treated explicitly.

Implicit--explicit discretizations are attractive because the viscous term can be treated implicitly while advection is evaluated explicitly, leaving only linear elliptic or Stokes-type solves. Foundational Runge--Kutta convergence theory for semilinear parabolic problems, including the NSE, was developed in \cite{lubich1996runge}; related maximal-regularity theory appears in \cite{kunstmann2018runge}. Accuracy and energy behavior of Runge--Kutta methods for incompressible flow, including differential--algebraic effects in velocity--pressure formulations, were studied in \cite{sanderse2012accuracy,sanderse2013energy}. Variable-step $G$-stability and second-order convergence of a fully implicit DLN discretization were established in \cite{layton2022dln}. For nonzero forcing, however, its stated energy estimate contains the accumulated forcing contribution and hence does not provide an absorbing uniform-in-time bound for persistently forced flows.

Rigorous convergence results for schemes with explicit or extrapolated advection have also been obtained in several settings. The fixed-step Crank--Nicolson/Adams--Bashforth method was analyzed in \cite{he2007cnab}, and optimal convergence of linearly extrapolated Crank--Nicolson methods was established in \cite{li2021second}. Low-regularity results include \cite{li2022critical,li2022semi}, while sharp variable-step IMEX-BDF2/SAV estimates for the unforced periodic NSE were derived in \cite{di2023variable}. For periodic NSE, high-order BDF-IMEX/SAV and related splitting schemes were analyzed in \cite{huang2022stability,huang2023stability,huang2025stability}. In the Runge--Kutta setting, \cite{li2025numerical} recently proved second-order fully discrete convergence for an IMEX-RK finite-element method with extrapolated convection for an unforced no-slip problem. These results provide important convergence benchmarks, but they do not simultaneously address a variable-step RK discretization, persistent forcing, explicit advection, and a direct uniform-in-time absorbing estimate.

Scalar auxiliary variable methods provide a flexible mechanism for constructing linearly implicit energy-stable schemes \cite{ShenXuYang2018SAV,SHenXuYang2019SAV}. For the NSE, several SAV-type methods exploit cancellation between the physical and scalar equations, including the ZEC-related schemes in \cite{yang2021new,yang2021novel_1,zhang2024unified,li2022new}. Related primitive-variable developments include the EMAC--ESAV schemes of \cite{lan2025robust}, the GePUP-ES formulation of \cite{li2025gepup}, and dynamically regularized Lagrange-multiplier methods \cite{doan2025dynamically,doan2026convergence,hoang2026error}. The concurrent-correction SAV construction of Hou and Qiao \cite{hou2023implicit} uses a first-order auxiliary variable to produce a higher-order correction of the physical variable. Mean reversion controls drift of the auxiliary variable in forced long-time simulations. Their combination, termed the mean-reverting concurrent-correction scalar auxiliary variable (mr-ccSAV) mechanism, was introduced for ETD discretizations in \cite{wang2026unconditionally}; the present work develops its Runge--Kutta counterpart.

Especially relevant is the work of Liao \textit{et al.}~\cite{liao2026longtime}. Their high-order IMEX-RK schemes use a neighboring-stage formulation, treat diffusion implicitly and advection explicitly, and yield uniform-in-time bounds for the periodic vorticity equation under a time-step restriction. Their analysis closes the $L^2$--$H^1$ estimates through a stagewise $H^\delta$ bootstrap, with the admissible time step depending on the resulting fractional-regularity bound and the system parameters. Related conditional long-time estimates for explicit-advection schemes were obtained in \cite{wang2012efficient,tone2015double,cheng2016long}. These arguments exploit the favorable vorticity or scalar-transport structure and do not directly yield a primitive-variable $H^1$ theory.

The present method has a different basic energy mechanism. The mr-ccSAV equation is coupled to the Runge--Kutta stages so that the explicitly evaluated advection cancels algebraically from the discrete enstrophy balance. Mean reversion then supplies damping for the scalar component. Consequently, the $L^2$ estimate closes directly, without a higher-regularity bootstrap, a time-step upper bound, or a neighboring-step-ratio restriction. For persistent forcing in $L^\infty(0,\infty;L^2)$, the forcing contribution remains bounded rather than accumulating with the terminal time, and the influence of the initial data decays. The mild assumption $\omega^0\in\dot L^s(\Omega)$, $s>2$, is used only to make the first explicit transport datum and the scalar stage coefficients well defined; the resulting $L^2$ bound depends on $\omega^0$ only through $|\omega^0|_{L^2}$. 
The uniform-in-time $L^2$ estimate also has a primitive-variable kinetic-energy analogue, although the detailed solvability, convergence, and numerical analysis in this paper are carried out in the periodic vorticity setting.

We propose an IMEX-SDIRK2-mr-ccSAV method based on Alexander's second-order SDIRK discretization of the viscous term and explicit extrapolation of advection. At each stage, an affine decomposition reduces the coupled problem to two elliptic solves with the same shifted Laplacian and the solution of either a cubic or a linear scalar algebraic equation. An important analytical ingredient is the positivity of the row-difference matrix $E^{-1}A^{\rm RK}$ associated with the implicit Butcher matrix $A^{\rm RK}$ \cite{liao2026longtime}.  Related row-difference positivity appeared earlier in \cite{shin2017unconditionally}. In the present setting, positivity of
\begin{equation}\label{eq}
S(E^{-1}A^{\rm RK})
:=
\frac12\left(
E^{-1}A^{\rm RK}+(E^{-1}A^{\rm RK})^T
\right)
\end{equation}
controls the viscous contribution, while the mr-ccSAV identity cancels the explicit advection. The convergence proof also uses the reference-stage-function approach of \cite{zhang2004error}.

The main results are as follows.
\begin{enumerate}
\item We construct a second-order IMEX-SDIRK2-mr-ccSAV method with explicit advection. Each stage requires two shifted elliptic solves and either a cubic or a linear scalar algebraic equation.
\item For $\omega^0\in\dot L^s(\Omega)$, $s>2$, a stage solution exists for every positive time step. Unique solvability is established separately under an explicit small-step condition.
\item For arbitrary positive time steps, we prove a direct unconditional $L^\infty(0,\infty;L^2)$ estimate on the vorticity. Under persistent bounded forcing, the estimate is absorbing and does not contain an accumulated forcing term.
\item Under additional regularity, a finite upper bound on the time steps, and bounded neighboring-step ratios, we establish uniform-in-time $H^1$ and $H^2$ vorticity bounds without a small-step condition.
\item For smooth forced solutions, we prove optimal second-order convergence on finite time intervals. Numerical experiments confirm accuracy preservation, long-time robustness, and effective time-step selection with a companion embedded scheme.
\end{enumerate}

To the best of our knowledge, this is the first variable-step RK-type discretization of the forced periodic NSE that treats advection explicitly and simultaneously admits guaranteed stagewise solvability, a direct unconditional uniform-in-time $L^2$ absorbing estimate, and a rigorous optimal temporal convergence rate.

The remainder of the paper is organized as follows. Section~\ref{sec:2} introduces the periodic vorticity setting and the mr-ccSAV reformulation. Section~\ref{sec:algorithm} presents the algorithm and its stagewise implementation. The subsequent sections establish solvability, uniform-in-time bounds, and finite-time convergence. Numerical experiments and concluding remarks follow.

\section{Preliminaries and the mean-reverting ccSAV reformulation of the NSE}\label{sec:2}

In this section, we first fix the functional setting for the periodic vorticity formulation of the incompressible Navier--Stokes equations. We then introduce the mean-reverting concurrent-correction SAV reformulation, which serves as the basis for the numerical scheme and for the subsequent stability and convergence analysis. The corresponding primitive-variable energy compatibility is discussed separately in Remark~\ref{rem:primitive-variable}; it is not part of the periodic solvability or convergence theory.

\subsection{Function spaces and notation}
We work on the periodic domain $\Omega=(0,2\pi)\times(0,2\pi)$ and use  zero-mean periodic Sobolev spaces. More specifically, let $\mathcal T_0$ denote the space of mean-zero trigonometric
polynomials on $\Omega=\mathbb T^2$.  For an integer $\ell\ge0$ and
$1<s<\infty$, define
\[
 \dot W_{\rm per}^{\ell,s}(\Omega)
 :=
 \overline{\mathcal T_0}^{\,W^{\ell,s}(\Omega)}.
\]
Thus the dot indicates zero spatial mean.  For negative integer orders,
we use the duality convention
\[
 \dot W_{\rm per}^{-\ell,s}(\Omega)
 :=
 \bigl(\dot W_{\rm per}^{\ell,s'}(\Omega)\bigr)',
 \qquad \frac1s+\frac1{s'}=1.
\]
In particular,
\[
 H=\dot W_{\rm per}^{0,2}(\Omega),\qquad
 V=\dot W_{\rm per}^{1,2}(\Omega),\qquad
 V'=\dot W_{\rm per}^{-1,2}(\Omega).
\]
We write $\dot L^s=\dot W_{\rm per}^{0,s}$.
Standard periodic Sobolev embeddings will be used below; see
\cite{AdamsFournier2003,Hebey2000}.

\ignore{
$W^{l,s}$. In particular, we define, for any real $s\ge 1$.
\[
\dot L^s(\Omega):=
\{\phi\in L^s(\Omega): \int_\Omega\phi=0
\}
\]
and, for any real $s\in \mathbb{N}^+$,
\[
\dot H_{\rm per}^s(\Omega)=\dot{W}^{s,2}_{\rm per}(\Omega):=
\left\{\phi\in H_{\rm loc}^s(\mathbb R^2):
2\pi\text{-per. in each direction},\int_\Omega\phi=0
\right\}.
\]
We denote by $\dot H_{\rm per}^{-1}(\Omega)$ the dual of
$\dot H_{\rm per}^1(\Omega)$. The notation $\langle\cdot,\cdot\rangle$
denotes either the $L^2$ inner product or, when appropriate, the induced
duality $\dot H_{\rm per}^{-1}$-
$\dot H_{\rm per}^1$ pairing, and $\|\cdot\|$ denotes the norm in $L^2(\Omega)$. We set
\begin{equation}
H:=\dot L^2(\Omega),
\qquad
V:=\dot H^1_{\rm per}(\Omega),
\qquad
V':=\dot H^{-1}_{\rm per}(\Omega).
\end{equation}
}

\subsection{Mean-reverting ccSAV reformulation}
Since $\psi=(-\Delta)^{-1}\omega$, define
\[
B(\omega)
:=\nabla^\perp\psi\cdot\nabla\omega
=\nabla^\perp(-\Delta)^{-1}\omega\cdot\nabla\omega.
\]
Then \eqref{eq:intro-nse} is $\omega_t-\nu\Delta\omega+B(\omega)=f$.
We introduce an auxiliary scalar $r(t)$ and the extended system
\begin{subequations}\label{eqn:mr_sav}
\begin{align}
    &\frac{\partial \omega}{\partial t} - \nu \Delta \omega + G_\omega(r)\,B(\omega) = f, \label{eqn:mr_sav_1_v2}\\
    &\frac{\mathrm{d} r}{\mathrm{d} t} + \gamma r + G_r(r) \langle B(\omega), \omega \rangle = 0, \label{eqn:mr_sav_2_v2}
\end{align}
\end{subequations}
where $\gamma>0$ is the mean-reversion parameter. 
For an integer $k\ge 1$, let
\begin{equation}\label{eqn:ccsav_factor}
G_\omega(r)=1-r^k,\qquad G_r(r) = \sum_{i = 0}^{k-1} r^i,
\end{equation}
so that $G_\omega(r)=(1-r)G_r(r)$.   
In this paper $k=2$, hence $G_\omega(r)=1-r^2$ and $G_r(r)=1+r$.

The periodic transport term satisfies
\[
 \langle B(\omega), \omega\rangle = 0,\qquad \forall\omega\in V.
\]
Consequently, the exact scalar equation is $r_t+\gamma r=0$, so
$r(t)=r(0)e^{-\gamma t}$ and $r\equiv0$ when $r(0)=0$. At the discrete
level, the damping suppresses drift away from this equilibrium. Meanwhile,
if $r=O(\tau)$, then $G_\omega(r)-1=-r^k=O(\tau^k)$, which makes the correction consistent with formal order-$k$ accuracy of the physical
variable. The combination of these mean-reverting and concurrent-correction
mechanisms, abbreviated as mr-ccSAV, was introduced for variable-step ETD
schemes in \cite{wang2026unconditionally}. It builds on the mean-reverting
SAV formulations of \cite{han2025highly,coleman2024efficient}, the
concurrent-correction idea of Hou and Qiao \cite{hou2023implicit}, and the zero-energy-preservation (ZEC) approach \cite{li2022new, yang2021novel, yang2021numerical}; see also
\cite{huang2022new} for related higher-order sequential correction strategies.

For context, $\gamma=0$ and $k=1$ recover the SAV-ZEC formulation used for
several fluid models \cite{yang2021new,yang2021novel,zhang2024unified,li2022new}.
The role of the seemingly redundant $\gamma>0$ here is to prevent accumulation of the auxiliary-variable
drift that may arise from explicit treatment of advection over long times.

\section{Algorithm} \label{sec:algorithm}
Building on the mr-ccSAV reformulation \eqref{eqn:mr_sav}, we construct a second-order singly diagonally implicit Runge--Kutta (SDIRK2) discretization in which the viscous and mean-reversion terms are treated implicitly, while the nonlinear advection term is treated explicitly. This yields the SDIRK2-mr-ccSAV scheme presented below. The solvability of the associated stage equations will be examined in detail in the next subsection.

\ignore{
Let
\[
\eta=1-\frac{1}{\sqrt2},
\qquad
\delta=1-\frac{1}{2\eta}.
\]
The implicit part is the standard stiffly accurate Alexander SDIRK2 method,
\begin{equation}\label{eq:standard-sdirk-tableau}
\begin{array}{c|cc}
\eta & \eta & 0\\
1 & 1-\eta & \eta\\
\hline
&1-\eta&\eta
\end{array},
\qquad
A^{\rm RK}:=
\begin{pmatrix}
\eta&0\\
1-\eta&\eta
\end{pmatrix}.
\end{equation}
The explicit companion evaluates the nonlinear and forcing terms at
$c_0=0$ and $c_1=\eta$ and has the standard tableau
\begin{equation}\label{eq:standard-explicit-tableau}
\begin{array}{c|ccc}
0&0&0&0\\
\eta&\eta&0&0\\
1&\delta&1-\delta&0\\
\hline
&\delta&1-\delta&0
\end{array}.
\end{equation}
For the two nontrivial stages, write
\[
\widehat A^{\rm RK}:=
\begin{pmatrix}
\eta&0\\
\delta&1-\delta
\end{pmatrix},
\qquad
E:=
\begin{pmatrix}
1&0\\
1&1
\end{pmatrix}.
\]
Following Liao et al \cite{liao2026longtime}, the neighboring-stage stability analysis uses the equivalent incremental
matrices
\begin{equation}\label{eq:incremental-rk-matrices}
A:=E^{-1}A^{\rm RK}
=\begin{pmatrix}
\eta&0\\
1-2\eta&\eta
\end{pmatrix},
\qquad
\widehat A:=E^{-1}\widehat A^{\rm RK}
=\begin{pmatrix}
\eta&0\\
\delta-\eta&1-\delta
\end{pmatrix}.
\end{equation}
Thus the standard and incremental forms describe the same IMEX method; the
latter merely rewrites each stage as an increment from its immediate
predecessor.}

{
\subsection{The SDIRK2-mr-ccSAV scheme}
\label{subsec:isdirk2_mr_ccsav}

Let
\[
\eta=1-\frac{1}{\sqrt{2}},
\qquad
\delta=1-\frac{1}{2\eta}.
\]
We first describe the underlying two-stage, second-order IMEX Runge--Kutta method in its standard form. 
The implicit part is the stiffly accurate Alexander SDIRK2 method, while
the explicit part evaluates the nonlinear and forcing terms at
$c_0=0$ and $c_1=\eta$. The corresponding IMEX pair and the coefficient matrices for the two nontrivial stages are
\begin{equation}\label{eq:standard-imex-tableau}
\begin{array}{c|ccc|ccc}
 & \multicolumn{3}{c|}{\text{implicit}}&
   \multicolumn{3}{c}{\text{explicit}}\\
0
 & 0      & 0        & 0
 & 0      & 0        & 0\\
\eta
 & 0      & \eta     & 0
 & \eta   & 0        & 0\\
1
 & 0      & 1-\eta   & \eta
 & \delta & 1-\delta & 0\\
\hline
 & 0      & 1-\eta   & \eta
 & \delta & 1-\delta & 0
\end{array},
\quad
A^{\mathrm{RK}}
=
\begin{pmatrix}
\eta&0\\
1-\eta&\eta
\end{pmatrix},
\quad
\widehat A^{\mathrm{RK}}
=
\begin{pmatrix}
\eta&0\\
\delta&1-\delta
\end{pmatrix}.
\end{equation} 

To express the method in terms of neighboring-stage increments, observe
that, for any stage sequence $\{Y_i\}_{i=0}^2$,
\[
\begin{pmatrix}
Y_1-Y_0\\
Y_2-Y_0
\end{pmatrix}
=
E
\begin{pmatrix}
Y_1-Y_0\\
Y_2-Y_1
\end{pmatrix},
\qquad
E:=
\begin{pmatrix}
1&0\\
1&1
\end{pmatrix}.
\]
Thus, subtracting each stage equation from its successor amounts to
multiplying the standard coefficient matrices by $E^{-1}$. Following the
neighboring-stage formulation of Liao {\it et al.}~\cite{liao2026longtime}, we
therefore introduce
\begin{equation}\label{eq:incremental-rk-matrices}
A:=E^{-1}A^{\mathrm{RK}}
=
\begin{pmatrix}
\eta&0\\
1-2\eta&\eta
\end{pmatrix},
\qquad
\widehat A:=E^{-1}\widehat A^{\mathrm{RK}}
=
\begin{pmatrix}
\eta&0\\
\delta-\eta&1-\delta
\end{pmatrix}.
\end{equation}
The standard and incremental formulations are therefore algebraically
equivalent; the latter is used below because it is more convenient for
neighboring-stage stability estimates.
}

Let $0=t_0<t_1<t_2<\cdots$ be an arbitrary temporal grid. For later
reference, define
\begin{equation}\label{eq:mesh-parameters}
\tau_n:=t_{n+1}-t_n,\qquad
\tau_*:=\sup_{n\ge0}\tau_n,\qquad
\rho_n:=\frac{\tau_{n+1}}{\tau_n},\qquad
\rho_*:=\sup_{n\ge0}\rho_n.
\end{equation}

On each interval $[t_n,t_{n+1}]$, the stage values
$\omega_{n,i}$ and $r_{n,i}$ approximate
$\omega(t_n+c_i\tau_n)$ and $r(t_n+c_i\tau_n)$, respectively, for
$i=0,1,2$. For notational convenience, set
\[
f_{n,j}:=f(t_n+c_j\tau_n),\quad
B_{n,j}:=B(\omega_{n,j})
=\nabla^\perp(-\Delta)^{-1}\omega_{n,j}\cdot\nabla\omega_{n,j},\quad j=0,1.
\]
\ignore{
Starting from $(\omega_{n,0},r_{n,0})=(\omega^n,r^n)$, the two internal stages are computed successively by the incremental IMEX Runge--Kutta scheme. More precisely, for $i=1,2$, we determine $(\omega_{n,i},r_{n,i})$ from
}
Starting from $(\omega_{n,0},r_{n,0})=(\omega^n,r^n)$, we compute the two
nontrivial stages successively using the incremental coefficients
$A=(a_{i,j})$ and $\widehat A=(\hat a_{i,j})$. More precisely, for
$i=1,2$, the stage values $(\omega_{n,i},r_{n,i})$ 
are determined by
\begin{subequations}\label{eqn:mr_sav_stage}
\begin{align}
\omega_{n,i} ={}& \omega_{n,i-1}
+\nu\tau_n\sum_{j=1}^{i}
 a_{i,j}\Delta\omega_{n,j}
+\tau_n\sum_{j=0}^{i-1}
\hat{a}_{i,j}
\bigl(f_{n,j}-G_\omega(r_{n,i})B_{n,j}\bigr), \label{eqn:isdirk2-mr-ccsav-omega}
\\
r_{n,i} ={}& r_{n,i-1} 
-\gamma\tau_n\sum_{j=1}^{i} a_{i,j}r_{n,j} -\tau_n\sum_{j=0}^{i-1} \hat a_{i,j}G_r(r_{n,i}) \langle B_{n,j},\omega_{n,i}\rangle, \label{eqn:isdirk2-mr-ccsav-r}
\end{align}
\end{subequations}
and the numerical solution at $t_{n+1}$ is defined by $(\omega^{n+1}, r^{n+1})=(\omega_{n,2}, r_{n,2})$, since the method is stiffly accurate. We refer to the resulting method as the IMEX SDIRK2-mr-ccSAV method, abbreviated as the {\bf SDIRK2-mr-ccSAV method}.

\subsection{Stagewise solvability}
In this subsection, we first decouple each coupled stage of
\eqref{eqn:mr_sav_stage} and derive a practical solution procedure
consisting of two linear elliptic problems with the same operator and
a scalar nonlinear equation. 
Based on this reduction, we establish stagewise existence for any $\tau_n>0$ and uniqueness under a sufficient step-size condition.

Fix a stage $i\in \{1,2 \}$ and suppose that all previous stage values are known. Define
\begin{equation}\label{eq:stage-explicit-data}
\widetilde f_{n,i}:=\sum_{j=0}^{i-1}\hat a_{i,j}f_{n,j},\qquad
\widetilde B_{n,i}:=\sum_{j=0}^{i-1}\hat a_{i,j}B_{n,j}.
\end{equation}
The $i$th stage vorticity equation can be written as
\begin{equation}\label{eq:stage-vorticity-solve}
\left(I-\nu\tau_n a_{i,i}\Delta\right)\omega_{n,i}
=R_{n,i}^\omega-\tau_nG_\omega(r_{n,i})\widetilde B_{n,i},
\end{equation}
where
\[
R_{n,i}^\omega:=
\omega_{n,i-1}+\nu\tau_n\sum_{j=1}^{i-1}a_{i,j}\Delta\omega_{n,j}
+\tau_n\widetilde f_{n,i}.
\]
Thus, the vorticity stage is decoupled by solving
\begin{equation}\label{eq:stage-two-solves}
\begin{aligned}
\left(I-\nu\tau_n  a_{i,i}\Delta\right)\omega_{n,i}^{(1)}=R_{n,i}^\omega,\quad \left(I-\nu\tau_n a_{i,i}\Delta\right)\omega_{n,i}^{(2)} =\widetilde B_{n,i}.
\end{aligned}
\end{equation}
For any scalar trial value $r$, the corresponding vorticity is
\begin{equation}\label{eq:stage-affine-vorticity}
\omega_{n,i}(r)=
\omega_{n,i}^{(1)}-\tau_nG_\omega(r)\omega_{n,i}^{(2)}.
\end{equation}

It remains to determine the scalar stage value $r_{n,i}$. Let
\[
R_{n,i}^r:=r_{n,i-1}-\gamma\tau_n\sum_{j=1}^{i-1}a_{i,j}r_{n,j}.
\]
The $r$-equation at the $i$th stage is
\begin{equation}\label{eq:stage-r-equation}
\left(1+\gamma\tau_n a_{i,i}\right)r_{n,i}
=
R_{n,i}^r-\tau_nG_r(r_{n,i}) \langle \widetilde B_{n,i},\omega_{n,i} \rangle.
\end{equation}
Define
\begin{equation}\label{eq:alpha-beta}
\alpha_{n,i}:= \langle \widetilde B_{n,i},\omega_{n,i}^{(1)} \rangle,
\qquad
\beta_{n,i}:= \langle \widetilde B_{n,i},\omega_{n,i}^{(2)} \rangle.
\end{equation}
Substituting \eqref{eq:stage-affine-vorticity} into
\eqref{eq:stage-r-equation} gives the scalar equation
\begin{equation}\label{eq:stage-scalar-equation}
\mathcal F_{n,i}(r):=
\left(1+\gamma\tau_n a_{i,i}\right)r-R_{n,i}^r
+\tau_nG_r(r)\left(\alpha_{n,i}-\tau_nG_\omega(r)\beta_{n,i}\right)=0.
\end{equation}
For the present second-order scheme, we set $k=2$ in
\eqref{eqn:ccsav_factor}, which gives $G_\omega(r)=1-r^2$ and
$G_r(r)=1+r$. Hence
\begin{subequations}\label{eq:stage-scalar-k2}
\begin{align}
\mathcal F_{n,i}(r)
&=
\left(1+\gamma\tau_n a_{i,i}\right)r-R_{n,i}^r
+\tau_n(1+r)\alpha_{n,i} + \tau_n^2(r^3+r^2-r-1)\beta_{n,i},\\
\mathcal F_{n,i}'(r)
&=
1+\gamma\tau_n a_{i,i}
+\tau_n\alpha_{n,i}+\tau_n^2\beta_{n,i}(3r^2+2r-1).
\end{align}
\end{subequations}
The existence and energy estimates below hold for any real root of
\eqref{eq:stage-scalar-k2}.  In the numerical implementation, Newton iteration
is initialized by the mean-reverting predictor
\begin{equation}\label{eq:root-predictor}
r_{n,i}^{(0)}:=\frac{R_{n,i}^r}{1+\gamma\tau_n a_{i,i}}.
\end{equation}
The iteration is safeguarded by a Newton--bisection hybrid whenever the
derivative becomes too small, the residual fails to decrease, or a Newton step
leaves the current bracket.  When $\beta_{n,i}>0$, the scalar polynomial has
odd degree and positive leading coefficient, so a bracket can always be found
by symmetrically expanding an interval centered at the predictor
\eqref{eq:root-predictor}.  When $\beta_{n,i}=0$,
the equation reduces to the linear equation described below.  Under the
sufficient small-step condition of
\cref{prop:stage-uniqueness-small-step}, the root is unique and hence
independent of the particular scalar solver.
Once a root $r_{n,i}$ has been obtained, the corresponding vorticity
$\omega_{n,i}$ is recovered directly from
\eqref{eq:stage-affine-vorticity}, after which
$B_{n,i}=B(\omega_{n,i})$ is evaluated for use in the subsequent stage.
Consequently, each stage requires two linear elliptic solves with the same operator and one scalar nonlinear solve.

\begin{algorithm}[H]
\caption{One-step SDIRK2-mr-ccSAV update}\label{alg:fixed-step}
\begin{algorithmic}[1]
\REQUIRE Accepted values $(\omega^n,r^n)$, step size $\tau_n>0$, and
forcing values $f_{n,0}$ and $f_{n,1}$.
\STATE Set $(\omega_{n,0},r_{n,0})=(\omega^n,r^n)$ and
$B_{n,0}=B(\omega^n)$.
\FOR{$i=1,2$}
  \STATE Form $\widetilde f_{n,i}$, $\widetilde B_{n,i}$,
  $R_{n,i}^\omega$, and $R_{n,i}^r$ from the already available stages.
  \STATE Solve the two shifted elliptic problems in
  \eqref{eq:stage-two-solves} for
  $\omega_{n,i}^{(1)}$ and $\omega_{n,i}^{(2)}$.
  \STATE Compute $\alpha_{n,i}$ and $\beta_{n,i}$ from
  \eqref{eq:alpha-beta}.
  \STATE Solve the scalar equation $\mathcal F_{n,i}(r)=0$ in
  \eqref{eq:stage-scalar-equation} by safeguarded Newton iteration initialized
  with the mean-reverting predictor \eqref{eq:root-predictor}.
  \STATE Set $r_{n,i}=r$ and recover $\omega_{n,i}$ from
  \eqref{eq:stage-affine-vorticity}.
  \IF{$i=1$}
    \STATE Evaluate $B_{n,1}=B(\omega_{n,1})$ for the second stage.
  \ENDIF
\ENDFOR
\STATE Accept $(\omega^{n+1},r^{n+1})=(\omega_{n,2},r_{n,2})$.
\end{algorithmic}
\end{algorithm}

For clarity, the complete computation on one time interval is summarized in
Algorithm~\ref{alg:fixed-step}.  Although the time grid may vary with $n$,
the term ``one-step'' emphasizes that only data from the current accepted
time level are needed.

\begin{proposition}[Stagewise existence and regularity]\label{prop:stage-solvability}
Assume that $\omega^n\in \dot{L}^s(\Omega)$ for some
$s>2$ and that $f_{n,j}\in V'$ for $j=0,1$.
Then, for
every $\tau_n>0$, each stage admits at least one solution
$\omega_{n,i}\in V$, $r_{n,i}\in\mathbb R, i=1,2$.

If, in addition, $\omega^n\in \dot H_{\rm per}^2(\Omega)$ and
$f_{n,j}\in H$ for $j=0,1$,
then $\omega_{n,i}\in \dot H_{\rm per}^2(\Omega)$ for $i=1,2$. If
$\omega^n\in \dot H_{\rm per}^3(\Omega)$ and $f_{n,j}\in V$ for
$j=0,1$,
then $\omega_{n,i}\in \dot H_{\rm per}^3(\Omega)$ for $i=1,2$. Consequently, the higher regularity
propagates from one accepted time level to the next.
\end{proposition}

\begin{proof}
For $s>2$, elliptic regularity and two-dimensional Sobolev embedding 
give 
$\bm u^n=\nabla^\perp(-\Delta)^{-1}\omega^n\in W^{1,s}(\Omega)\subset L^\infty(\Omega)$. Hence,
\[
\|B(\omega^n)\|_{V'}
\le \|\omega^n\|\,\|\bm u^n\|_{L^\infty}
\le C\|\omega^n\|\,\|\omega^n\|_{L^s}.
\]
Hence, $\widetilde B_{n,1}\in V'$ and
$R^\omega_{n,1}\in V'$. Since $a_{1,1}>0$, the operator
$I-\nu\tau_n a_{1,1}\Delta:V\to V'$ is invertible, so
$\omega_{n,1}^{(j)}\in V$ for $j=1,2$. This implies, once $r_{n,1}$
has been selected, that $\omega_{n,1}\in V$. In turn,
$\widetilde B_{n,2}\in V', R^\omega_{n,2}\in V'$. 
Hence the two elliptic problems
\eqref{eq:stage-two-solves} are well posed in $V$.

It remains to establish the solvability of the scalar equation.
Since $I-\nu\tau_n a_{i,i}\Delta: V\to V'$ is positive definite on the mean-zero space, its inverse is also positive definite, and
$(I-\nu\tau_n a_{i,i}\Delta)^{-1}g\in V$ if $g\in V'$. Consequently,
\[
\beta_{n,i} =
\left\langle
\widetilde B_{n,i},
\left(I-\nu\tau_n a_{i,i}\Delta\right)^{-1}\widetilde B_{n,i}
\right\rangle
\ge 0,
\]
where the brackets denote the $V'-V$ duality pairing. The quantity is
finite and equality holds if and only if $\widetilde B_{n,i}=0$.

Suppose first that $\beta_{n,i}>0$. Equation \eqref{eq:stage-scalar-k2} implies that $\mathcal F_{n,i}$ is a cubic polynomial with a nonzero positive
leading coefficient $\tau_n^2\beta_{n,i}$.
\ignore{
\[
\mathcal F_{n,i}(r)\to-\infty\quad\text{as }r\to-\infty,
\qquad
\mathcal F_{n,i}(r)\to+\infty\quad\text{as }r\to+\infty.
\]
}
The intermediate value theorem therefore guarantees at least one real root.

If $\beta_{n,i}=0$, then $\widetilde B_{n,i}=0$ and, consequently, $\alpha_{n,i}=0$ according to \eqref{eq:alpha-beta}. In this
case, \eqref{eq:stage-scalar-equation} reduces to
$
\left(1+\gamma\tau_n a_{i,i}\right)r=R_{n,i}^r,
$
which has a unique real solution because
$1+\gamma\tau_n a_{i,i}>0$.

Thus, in either case, the scalar equation admits at least one real root.  The
existence statement permits any such root; the safeguarded Newton procedure
described above specifies the root used in the computations.
Once a root $r_{n,i}$ has been obtained,
\eqref{eq:stage-affine-vorticity} uniquely determines $\omega_{n,i}\in V$. Hence the $i$th stage admits at least one solution in $V$.

For the higher-regularity assertions, the two-dimensional product estimates
give
\[
B:H^2(\Omega)\to L^2(\Omega),
\qquad B:H^3(\Omega)\to H^1(\Omega).
\]
Periodic elliptic regularity for
$I-\nu\tau_n a_{i,i}\Delta$ therefore yields the asserted $H^2$ and $H^3$
stage regularity. Stiff accuracy then propagates this regularity to the next
accepted time level.
\end{proof}

\begin{remark}
The assumption $\omega^0\in \dot L^s$, $s>2$, is used only to ensure
that the first explicit advection datum belongs to $V'$ and hence that
$\alpha_{n,i}$ and $\beta_{n,i}$ are finite.
 The additional
$L^s$ integrability does not enter the uniform-in-time $L^2$ energy bound,
whose dependence on the initial data is only through
$\|\omega^0\|_{L^2}$.
\end{remark}

\begin{proposition}[Small-step stagewise uniqueness]
\label{prop:stage-uniqueness-small-step}
Assume that there exist constants $M_\alpha,M_\beta\geq0$ such that the $\alpha_{n,i},\beta_{n,i}$ defined in \eqref{eq:alpha-beta} satisfy
\begin{equation}\label{eq:alpha-beta_bound}
|\alpha_{n,i}|\le M_\alpha,\qquad 0\le\beta_{n,i}\le M_\beta.
\end{equation}
Then the $i$th stage equation \eqref{eqn:mr_sav_stage} is uniquely solvable under the following small-step condition:
\begin{equation}\label{eq:time_step_restriction}
\tau_nM_\alpha+\frac43\tau_n^2M_\beta<1+\gamma\tau_n a_{i,i}.
\end{equation}
\end{proposition}

\begin{proof}
Equation \eqref{eq:stage-scalar-k2} gives
\[
\mathcal F_{n,i}'(r)
=
1+\gamma\tau_n a_{i,i}
+\tau_n\alpha_{n,i}
+\tau_n^2\beta_{n,i}(3r^2+2r-1).
\]
Since $3r^2+2r-1\ge -\frac43$, for  $r\in\mathbb R$, we obtain
\[
\begin{aligned}
\mathcal F_{n,i}'(r)
&\ge
1+\gamma\tau_n a_{i,i}-\tau_n|\alpha_{n,i}|-\frac43\tau_n^2\beta_{n,i} 
\ge
1+\gamma\tau_n a_{i,i}-\tau_nM_\alpha-\frac43\tau_n^2M_\beta.
\end{aligned}
\]
Hence $\mathcal F_{n,i}'(r)>0$ provided that \eqref{eq:time_step_restriction} is satisfied.
Thus the scalar equation has a unique root, and
\eqref{eq:stage-affine-vorticity} then uniquely determines the vorticity
stage.
\end{proof}

\subsection{Embedded adaptive time-stepping algorithm}
In this subsection, we propose an embedded adaptive time-stepping strategy based on the SDIRK2-mr-ccSAV scheme. 
The strategy uses previously accepted numerical solutions and the first SDIRK stage to construct an embedded error indicator. If the indicator exceeds the prescribed tolerance, the current time step is rejected, reduced, and recomputed.

Since a local Taylor expansion gives
$\omega_{n,1}=\omega(t_n+c_1\tau_n)+O(\tau_n^2)$,
linear Lagrange extrapolation from $t_n$ and $t_n+c_1\tau_n$ yields the
embedded pair
\begin{equation}\label{eq:adaptive-extrapolation}
\omega_{(1)}^{n+1}
:=
\frac{c_1-1}{c_1}\omega^n
+\frac{1}{c_1}\omega_{n,1},
\qquad
\omega_{(2)}^{n+1}
:=
\omega_{n,2}
=
\omega^{n+1}.
\end{equation}
Here, $\omega_{(1)}^{n+1}$ and $\omega_{(2)}^{n+1}$ denote the
first- and second-order approximations at $t_{n+1}$, respectively.
Then, the embedded error indicator is defined by
\begin{equation}\label{eq:adaptive-embedded-error}
    \begin{aligned}
    e_\omega^{n+1} =
    \frac{\|\omega_{(2)}^{n+1}-\omega_{(1)}^{n+1}\|}
    {\max\{\|\omega_{(2)}^{n+1}\|,\varepsilon_{\rm ref}\}},\qquad
    e_r^{n+1}=|r^{n+1}|.
    \end{aligned}
\end{equation}
Here $\varepsilon_{\rm ref}>0$ is a fixed normalization floor that prevents
division by zero when the vorticity norm is very small.



Following \cite{hairer1993solving}, the adaptive strategy is summarized in
\cref{alg:adaptive-extrapolation}. The time-step update function is defined by
\begin{equation}\label{eq:adaptive-step-update}
A_{\mathrm{ext}}(e_\omega,e_r,\tau)
=
\rho
\min\left\{
\left(\frac{\mathrm{tol}_\omega}{e_\omega}\right)^{1/2},
\frac{\mathrm{tol}_r}{e_r}
\right\}\tau,
\end{equation}
where $\rho\in(0,1)$ is a safety factor, and
$\mathrm{tol}_\omega$ and $\mathrm{tol}_r$ are prescribed tolerances for
$\omega$ and $r$, respectively. If an error indicator vanishes, the
corresponding quotient in \eqref{eq:adaptive-step-update} is interpreted as
$+\infty$.

Since $\omega_{(2)}^{n+1}$ and $\omega_{(1)}^{n+1}$ are second- and
first-order approximations, respectively, their difference is of order
$O(\tau_n^2)$ under sufficient temporal regularity. This motivates the
exponent $1/2$ in \eqref{eq:adaptive-step-update}. Because the SDIRK2
scheme is a one-step method, the adaptive controller can be activated
from the first time step using a prescribed initial trial step.

\begin{algorithm}[ht]
\footnotesize
\caption{Extrapolation-based adaptive SDIRK2-mr-ccSAV method}
\label{alg:adaptive-extrapolation}
\begin{algorithmic}
\STATE \textbf{Given:}
$\omega^n$, $r^n$, and a trial step $\tau_n$.
\STATE \textbf{Step 1.}
Compute $(\omega^{n+1},r^{n+1}, \omega_{n,1})$ from $(\omega^n,r^n)$ by the
SDIRK2-mr-ccSAV scheme with time step $\tau_n$.
\STATE \textbf{Step 2.}
Compute $\omega_{\mathrm{(1)}}^{n+1}$ from
\eqref{eq:adaptive-extrapolation}.
\STATE \textbf{Step 3.}
Compute $e_\omega^{n+1}$ and $e_r^{n+1}$.
\IF{$e_\omega^{n+1}\le \mathrm{tol}_\omega$ \textbf{and}
$e_r^{n+1}\le \mathrm{tol}_r$}
\STATE \textbf{Step 4.}
Accept $(\omega^{n+1},r^{n+1})$ and update
\[
\tau_{n+1}
\leftarrow
\max\left\{\tau_{\min},
\min\left\{A_{\mathrm{ext}}(e_\omega^{n+1},e_r^{n+1},\tau_n),
\tau_{\max}\right\}\right\}.
\]
\ELSE
\IF{$\tau_n=\tau_{\min}$}
\STATE Stop and report that the prescribed tolerances cannot be met at
$\tau_{\min}$.
\ELSE
\STATE \textbf{Step 5.}
Reject $(\omega^{n+1},r^{n+1})$ and reset
\[
\tau_n
\leftarrow
\max\left\{\tau_{\min},
\min\left\{A_{\mathrm{ext}}(e_\omega^{n+1},e_r^{n+1},\tau_n),
\tau_{\max}\right\}\right\}.
\]
\STATE \textbf{Step 6.} Return to \textbf{Step 1}.
\ENDIF
\ENDIF
\end{algorithmic}
\end{algorithm}

\section{Uniform-in-Time Stability}

We first collect the algebraic and analytic tools used below. The uniform-in-time estimates for the variable-step SDIRK2-mr-ccSAV approximation are proved in the following subsections.
Since the considered vorticity variable has zero mean, the Poincar\'e
inequality for the mean-zero periodic functions yields
\begin{equation}\label{eqn:poincare}
\|h\|^2
\le
C_P\|\nabla h\|^2,
\qquad \forall h\in V,
\end{equation}
where the constant $C_P$ depends only on the domain $\Omega$. 
The incremental Runge-Kutta coefficent matrices $A, \hat{A}$ are introduced earlier in \eqref{eq:incremental-rk-matrices}.
\ignore{
We next introduce the incremental Runge--Kutta coefficient matrices 
\begin{equation*}
\begin{aligned}
A
:=
\begin{pmatrix}
a_{1,1} & 0 \\
a_{2,1} & a_{2,2}
\end{pmatrix}
=
\begin{pmatrix}
\eta & 0 \\
1-2\eta & \eta
\end{pmatrix},\\[3pt]
\hat A
:=
\begin{pmatrix}
\hat a_{1,0} & 0 \\
\hat a_{2,0} & \hat a_{2,1}
\end{pmatrix}
=
\begin{pmatrix}
\eta & 0 \\
\delta-\eta & 1-\delta
\end{pmatrix}.
\end{aligned}
\end{equation*}
}

\begin{lemma}[Stage coefficient bounds {\cite[Lemma~2.2]{liao2026longtime}}]
    \label{lem:stage-coefficient-bounds}
Let $\lambda_I$ and $\sigma_E$ be the constants defined by
\[
\lambda_I
=\lambda_{\min}\!\left(\frac{A+A^T}{2}\right)
\approx 0.0858,\qquad
\sigma_E
=\|\hat A\|_2
\approx 1.9841.
\]
Then, for any space-time sequences
$\{v_i,u_i\}_{i\ge1}$, the following estimates hold:
\begin{equation}\label{eq:implicit-stage-coercivity}
\sum_{i=1}^{k}\sum_{j=1}^{i}
a_{i,j}\langle u_j,u_i\rangle \ge \lambda_I\sum_{i=1}^{k}\|u_i\|^2,
\end{equation}
and
\begin{equation}\label{eq:explicit-stage-bound}
\left|
\sum_{i=1}^{k}\sum_{j=0}^{i-1}
\hat a_{i,j}\langle v_j,u_i\rangle
\right|
\le
\sigma_E
\left(\sum_{j=0}^{k-1}\|v_j\|^2\right)^{1/2}
\left(\sum_{i=1}^{k}\|u_i\|^2\right)^{1/2}.
\end{equation}
\end{lemma}

We next present two discrete versions of Gronwall's inequality.

\begin{lemma}[Damped discrete Gronwall inequality]\label{lem:damped-discrete-gronwall}
Let $\alpha>0$, $C\ge0$, and let $\{\tau_n\}_{n\ge0}$ be a given positive time-step sequence.
Assume that the nonnegative stage values $\{v_{n,k}\}$ satisfy
\[
v_{n,0}=v^n,\qquad v^{n+1}=v_{n,s},
\]
where $s\ge1$ is a fixed integer. Suppose that for every $n\ge0$ and $1\le k\le s$,
\begin{equation}\label{eq:damped-gronwall-stage}
(1+\alpha\tau_n)v_{n,k}\le v_{n,0}+C\tau_n.
\end{equation}
Then,
\begin{equation}\label{eq:damped-gronwall-bound}
v_{n,k}
\le
\frac{1}{1 + \alpha t_{n+1}} v^0+ \frac{C}{\alpha},
\qquad n\ge0,\quad 1\le k\le s.
\end{equation}
\end{lemma}

\begin{proof}
Dividing \eqref{eq:damped-gronwall-stage} by
$1+\alpha\tau_n$, we obtain
\begin{equation}\label{eq:damped-gronwall-stage-rewrite}
v_{n,k}
\le
\frac{1}{1+\alpha\tau_n}v^n
+
\frac{C}{\alpha}
\left(1-\frac{1}{1+\alpha\tau_n}\right).
\end{equation}
Taking $k=s$ and iterating the resulting endpoint recursion, followed
by one further application of the stage estimate, yields
\begin{equation}\label{eq:damped-gronwall-endpoint}
v_{n,k}
\le
\left(\prod_{\ell=0}^{n}
(1+\alpha\tau_\ell)^{-1}\right)v^0
+
\frac{C}{\alpha}
\left[
1-
\prod_{\ell=0}^{n}
(1+\alpha\tau_\ell)^{-1}
\right],
\qquad n\ge0,
\end{equation}
where the empty product is understood as one. 

Moreover, since $\alpha,\tau_\ell>0$ and
$\sum_{\ell=0}^{n}\tau_\ell=t_{n+1}$,
\[
\prod_{\ell=0}^{n}(1+\alpha\tau_\ell)
\ge
1+\alpha\sum_{\ell=0}^{n}\tau_\ell
=
1+\alpha t_{n+1}.
\]
Using $v^0\ge0$, we therefore obtain
\[
v_{n,k}
\le
\frac{1}{1+\alpha t_{n+1}}v^0+\frac{C}{\alpha},
\]
which proves \eqref{eq:damped-gronwall-bound}.
\end{proof}

\begin{lemma}[Sliding-window damped discrete Gronwall inequality]
\label{lem:discrete-gronwall}
Let $\alpha>0$, and let $\{\tau_n\}_{n\ge0}$ denote a sequence of positive
time steps such that $\tau_*:=\sup_{n\ge0}\tau_n$ is bounded.
Set $t_0=0$ and for each $n\ge1$ define $t_n:=\sum_{\ell=0}^{n-1}\tau_\ell$.

Let $\{v^n\}_{n\ge0}$ and $\{g_n\}_{n\ge0}$ be nonnegative
sequences satisfying
\begin{equation}\label{eq:uniform-gronwall-step}
(1+\alpha\tau_n)v^{n+1}
\le v^n+\tau_ng_n,
\qquad n\ge0.
\end{equation}
Assume that, for every $T>0$, there exists a constant $C_T$ such
that
\[
\sup_{m\ge0}
\sup_{\substack{N>m\\ t_N-t_m\le T}}
\sum_{p=m}^{N-1}\tau_p g_p
\le C_T.
\]
Then, for every fixed $T_*>\tau_*$ and every $n\ge0$,
\begin{equation}\label{eq:uniform-gronwall-bound}
v^n
\le
\frac{v^0}{1+\alpha t_n}
+
\left(
1+\frac{1}{\alpha(T_*-\tau_*)}
\right)C_{T_*}.
\end{equation}
\end{lemma}

\begin{proof}
Iterating \eqref{eq:uniform-gronwall-step} gives
\begin{equation}\label{eq:uniform-gronwall-iteration}
v^n
\le
\left(
\prod_{\ell=0}^{n-1}(1+\alpha\tau_\ell)^{-1}
\right)v^0
+
\sum_{p=0}^{n-1}
\left(
\prod_{\ell=p}^{n-1}(1+\alpha\tau_\ell)^{-1}
\right)\tau_p g_p.
\end{equation}
Since $\prod_{\ell=0}^{n-1}(1+\alpha\tau_\ell) \ge 1+\alpha t_n$, the first term on the right-hand side of \eqref{eq:uniform-gronwall-iteration} is bounded by $v^0/(1+\alpha t_n)$.

It remains to estimate the accumulated source term. 
Starting from $N_0=n$, partition the indices backward into consecutive
blocks
\[
B_r:=\{N_{r+1},\ldots,N_r-1\},
\]
where $N_{r+1}$ is chosen as the smallest nonnegative integer satisfying
$t_{N_r}-t_{N_{r+1}}\le T_*$.
Except possibly for the last block, the maximality of each block and
$\tau_n\le\tau_*$ imply
\[
T_* - \tau_* <t_{N_r}-t_{N_{r+1}}\le T_*.
\]
Consequently, for every complete block $B_r$, we have
\[
\prod_{\ell=N_{r+1}}^{N_r-1}
(1+\alpha\tau_\ell)^{-1}
\le
\frac{1}{1+\alpha(t_{N_r}-t_{N_{r+1}})}
\le 
\frac{1}{1+\alpha (T_*- \tau_*)}.
\]

If $p\in B_r$, the product from $p$ to $n-1$ contains the $r$
complete blocks $B_0,\ldots,B_{r-1}$, and hence
\[
\prod_{\ell=p}^{n-1}(1+\alpha\tau_\ell)^{-1}
\le \frac{1}{(1+\alpha (T_*- \tau_*))^r}.
\]
Moreover, each block has temporal length at most $T_*$, so the
sliding-window assumption gives $\sum_{p\in B_r}\tau_p g_p\le C_{T_*}$. Therefore,
\begin{equation*}
\begin{aligned}
\sum_{p=0}^{n-1}
\left(
\prod_{\ell=p}^{n-1}(1+\alpha\tau_\ell)^{-1}
\right)\tau_p g_p 
&\leq
\sum_{r=0}^{\infty} \sum_{p\in B_r}
\left(
\prod_{\ell=p}^{n-1}(1+\alpha\tau_\ell)^{-1}
\right)\tau_p g_p \\
&\le
\sum_{r=0}^{\infty}\frac{1}{(1+\alpha (T_*- \tau_*))^r}\sum_{p\in B_r}\tau_p g_p \\
&=
\left( 1+\frac{1}{\alpha(T_*-\tau_*)} \right)C_{T_*}.
\end{aligned}
\end{equation*}
Combining this estimate with
\eqref{eq:uniform-gronwall-iteration} proves
\eqref{eq:uniform-gronwall-bound}.
\end{proof}

\subsection{Uniform-in-time enstrophy (\texorpdfstring{$L^2$}{L2} norm) estimate}

We first present a uniform-in-time bound on the enstrophy
\begin{equation}\label{eqn:enstrophy_def}
    \mathcal{E}(t) := \frac{1}{2}\|\omega(\cdot,t)\|^2.
\end{equation}
Uniform-in-time enstrophy bound is the same as an $L^\infty(0,\infty; L^2)$ estimate on the vorticity.
\begin{theorem}\label{lem:l2-uniform-bound}
Assume $\gamma>0$ and that $\omega^0\in \dot L^s(\Omega)$ for some $s>2$,
$r^0\in\mathbb R$, and $f\in L^\infty(0,\infty;L^2(\Omega))$. Assume also that the stage values are
chosen so that
\[
\|f_{n,j}\|^2\le C_f:=\|f\|^2_{L^\infty(0,\infty;L^2(\Omega))},
\qquad n\ge0,\quad 0\le j\le 1.
\]
Then the SDIRK2-mr-ccSAV solution satisfies
\begin{equation}\label{eq:L2-stage-bound}
\|\omega_{n,k}\|^2+|1-r_{n,k}|^2 \le \frac{1}{1 + \alpha t_{n+1}} (\|\omega^0\|^2+|1-r^0|^2) + \frac{C_1}{\alpha},
\end{equation}
for every $n\ge0$ and $1\le k\le2$, where
\begin{equation*}
\alpha:=\min\left\{{\nu\lambda_I}/{C_P},\gamma\lambda_I\right\}, \qquad
C_1:=\frac{2\sigma_E^2C_PC_f}{\nu\lambda_I}+\frac{\gamma}{\lambda_I}.
\end{equation*}
\end{theorem}

\begin{proof}
By Proposition~\ref{prop:stage-solvability}, the initial $L^s$ regularity
makes the first step well defined. Each implicit stage belongs to $H^1\subset L^s$, so
the same solvability argument applies inductively at every subsequent step.

Step 1. [Stagewise energy identity]
Let $E_{n,i}:=\|\omega_{n,i}\|^2+|1-r_{n,i}|^2$.
Taking the $L^2$ inner product of the vorticity equation
\eqref{eqn:isdirk2-mr-ccsav-omega} with
$2\omega_{n,i}$ gives
\begin{equation}\label{eq:l2-omega-stage}
\begin{aligned}
&\|\omega_{n,i}\|^2-\|\omega_{n,i-1}\|^2
+\|\omega_{n,i} - \omega_{n,i-1}\|^2
+2\nu\tau_n\sum_{j=1}^{i}
a_{i,j}\langle\nabla\omega_{n,j},\nabla\omega_{n,i}\rangle\\
&=
2\tau_n\sum_{j=0}^{i-1}
\hat a_{i,j}\langle f_{n,j},\omega_{n,i}\rangle
-2\tau_nG_\omega(r_{n,i})
\sum_{j=0}^{i-1}\hat a_{i,j}
\langle B_{n,j},\omega_{n,i}\rangle.
\end{aligned}
\end{equation}
Similarly, multiplying the scalar equation \eqref{eqn:isdirk2-mr-ccsav-r} by $2(r_{n,i}-1)$ 
yields
\begin{equation}\label{eq:l2-r-stage}
\begin{aligned}
&|1-r_{n,i}|^2-|1-r_{n,i-1}|^2 +|r_{n,i} - r_{n,i-1}|^2 +2\gamma\tau_n\sum_{j=1}^{i} a_{i,j}(r_{n,i}-1)(r_{n,j} - 1)\\
&= 2\tau_nG_\omega(r_{n,i}) \sum_{j=0}^{i-1}
\hat a_{i,j}\langle B_{n,j},\omega_{n,i}\rangle +2\gamma\tau_n\sum_{j=1}^{i} a_{i,j}(1-r_{n,i}),
\end{aligned}
\end{equation}
where the identity $G_\omega=(1-r)G_r$ has been used. Adding
\eqref{eq:l2-omega-stage} and \eqref{eq:l2-r-stage} cancels the nonlinear
term. Summing the result over $i=1,\ldots,k$, with $k=1$ or $2$, gives
\begin{equation}\label{eq:l2-stage-sum}
\begin{aligned}
&E_{n,k}-E_{n,0}
+\sum_{i=1}^{k}\|\omega_{n,i} - \omega_{n,i-1}\|^2
+\sum_{i=1}^{k}|r_{n,i} - r_{n,i-1}|^2\\
&\quad
+2\nu\tau_n\sum_{i=1}^{k}\sum_{j=1}^{i}
a_{i,j}\langle\nabla\omega_{n,j},\nabla\omega_{n,i}\rangle
+2\gamma\tau_n\sum_{i=1}^{k}\sum_{j=1}^{i}
a_{i,j}(r_{n,i}-1)(r_{n,j}-1)\\
&=
2\tau_n\sum_{i=1}^{k}\sum_{j=0}^{i-1}
\hat a_{i,j}\langle f_{n,j},\omega_{n,i}\rangle
+2\gamma\tau_n\sum_{i=1}^{k}\sum_{j=1}^{i} a_{i,j}(1-r_{n,i}).
\end{aligned}
\end{equation}

Step 2. [Coercivity of the dissipative terms]
We next estimate the terms on the left-hand side of
\eqref{eq:l2-stage-sum}. Applying
\cref{lem:stage-coefficient-bounds}, we obtain
\begin{equation}\label{eq:l2-implicit-omega}
2\nu\tau_n\sum_{i=1}^{k}\sum_{j=1}^{i}
a_{i,j}\langle\nabla\omega_{n,j},\nabla\omega_{n,i}\rangle
\ge
2\nu\lambda_I\tau_n\sum_{i=1}^{k}\|\nabla\omega_{n,i}\|^2.
\end{equation}
Similarly,
\begin{equation}\label{eq:l2-implicit-r}
2\gamma\tau_n\sum_{i=1}^{k}\sum_{j=1}^{i}
a_{i,j}(r_{n,i}-1)(r_{n,j}-1)
\ge
2\gamma\lambda_I\tau_n\sum_{i=1}^{k}|1-r_{n,i}|^2.
\end{equation}

Step 3. [Bounds for the linear scalar term and external forcing]
By the Cauchy--Schwarz inequality and Young's inequality,
\begin{equation}\label{eq:l2-r-linear-bound}
\begin{aligned}
2\gamma\tau_n
\left|\sum_{i=1}^{k}\sum_{j=1}^{i} a_{i,j}(1-r_{n,i})\right|
&\le 2\gamma\tau_n
\left(\sum_{i=1}^{k}|1-r_{n,i}|^2\right)^{1/2}\\
&\le \gamma\lambda_I\tau_n\sum_{i=1}^{k}|1-r_{n,i}|^2
+\frac{\gamma}{\lambda_I}\tau_n,
\end{aligned}
\end{equation}
Here we used
$\max_{1\le k\le2} \left(\sum_{i=1}^{k}\left|\sum_{j=1}^{i}a_{i,j}\right|^2 \right)^{1/2}\le 1$.

For the external forcing term, \cref{lem:stage-coefficient-bounds}, Poincaré
inequality \eqref{eqn:poincare}, and Young's inequality imply
\begin{equation}\label{eq:l2-force-bound}
\begin{aligned}
2\tau_n\sum_{i=1}^{k}\sum_{j=0}^{i-1}
\hat a_{i,j}\langle f_{n,j},\omega_{n,i}\rangle \le &
2\sigma_E\tau_n
\left(\sum_{j=0}^{k-1}\|f_{n,j}\|^2\right)^{1/2}
\left(\sum_{i=1}^{k}\|\omega_{n,i}\|^2\right)^{1/2}\\
\le &
\nu\lambda_I\tau_n\sum_{i=1}^{k}\|\nabla\omega_{n,i}\|^2
+\frac{\sigma_E^2C_P}{\nu\lambda_I}\tau_n
\sum_{j=0}^{k-1}\|f_{n,j}\|^2.
\end{aligned}
\end{equation}
Substituting
\eqref{eq:l2-implicit-omega}--\eqref{eq:l2-r-linear-bound} and
\eqref{eq:l2-force-bound} into \eqref{eq:l2-stage-sum}, and discarding the
nonnegative increment terms, yields
\begin{equation}\label{eq:l2-stage-dissipative}
\begin{aligned}
& E_{n,k}-E_{n,0}
+\nu\lambda_I\tau_n\sum_{i=1}^{k}\|\nabla\omega_{n,i}\|^2
+\gamma\lambda_I\tau_n\sum_{i=1}^{k}|1-r_{n,i}|^2 \\
\le & \frac{\sigma_E^2C_P}{\nu\lambda_I}\tau_n
\sum_{j=0}^{k-1}\|f_{n,j}\|^2+\frac{\gamma}{\lambda_I}\tau_n.
\end{aligned}
\end{equation}

Step 4. [Uniform-in-time estimate]
Let $\alpha:=\min\left\{\nu\lambda_I/C_P,\gamma\lambda_I\right\}$,
and set $E^n:=\|\omega^n\|^2+|1-r^n|^2$. Poincare's inequality implies, for
each $1\le k\le2$,
\[
\nu\lambda_I\sum_{i=1}^{k}\|\nabla\omega_{n,i}\|^2
+\gamma\lambda_I\sum_{i=1}^{k}|1-r_{n,i}|^2
\ge
\alpha E_{n,k}.
\]
Since $\sum_{j=0}^{k-1}\|f_{n,j}\|^2\le 2C_f$ for $k=1,2$,
\eqref{eq:l2-stage-dissipative} gives
\[
\left(1+\alpha\tau_n\right)E_{n,k}\le E_{n,0}+C_1\tau_n,
\qquad
C_1:=\frac{2\sigma_E^2C_PC_f}{\nu\lambda_I}+\frac{\gamma}{\lambda_I}.
\]
Applying \cref{lem:damped-discrete-gronwall}, we obtain
\begin{equation*}
E_{n,k} \le \frac{1}{1 + \alpha t_{n+1}}E^0+\frac{C_1}{\alpha}.
\end{equation*}
The proof is complete.
\end{proof}

\begin{corollary}[Uniform $H^1$ sliding-window stage dissipation]\label{cor:sliding-H1}
Under the assumptions of \cref{lem:l2-uniform-bound},
for every fixed window length $T_*>0$,
\begin{equation}\label{eq:sliding-H1}
\sup_{m\ge0}\;\sup_{\substack{N>m\\\sum_{n=m}^{N-1}\tau_n\le T_*}}
\sum_{n=m}^{N-1}\tau_n\sum_{i=1}^{2}
\|\nabla\omega_{n,i}\|^2
\le
\frac{E^0+C_1/\alpha+C_1T_*}{\nu\lambda_I},
\end{equation}
where $E^0:=\|\omega^0\|^2+|1-r^0|^2$, and $\alpha$ and $C_1$ are defined in
\cref{lem:l2-uniform-bound}. In particular, the right-hand side is independent
of the starting index $m$ and of the variable time-step sequence.
\end{corollary}

\begin{proof}
Taking $k=2$ in \eqref{eq:L2-stage-bound} gives
\[
E^n=\|\omega^n\|^2+|1-r^n|^2
\le E^0+\frac{C_1}{\alpha},
\qquad n\ge0.
\]
Taking $k=2$ in \eqref{eq:l2-stage-dissipative} and using
$\sum_{j=0}^{1}\|f_{n,j}\|^2\le2C_f$ gives
\[
\nu\lambda_I\tau_n\sum_{i=1}^{2}\|\nabla\omega_{n,i}\|^2
+\gamma\lambda_I\tau_n\sum_{i=1}^{2}|1-r_{n,i}|^2
\le
E^n-E^{n+1}+C_1\tau_n.
\]
Summing from $n=m$ to $N-1$ yields
\[
\nu\lambda_I\sum_{n=m}^{N-1}\tau_n\sum_{i=1}^{2}\|\nabla\omega_{n,i}\|^2
+\gamma\lambda_I\sum_{n=m}^{N-1}\tau_n\sum_{i=1}^{2}|1-r_{n,i}|^2
\le
E^m-E^N+C_1\sum_{n=m}^{N-1}\tau_n .
\]
Since $E^N\ge0$, $E^m\le E^0+C_1/\alpha$, and
$\sum_{n=m}^{N-1}\tau_n\le T_*$, we obtain
\[
\nu\lambda_I\sum_{n=m}^{N-1}\tau_n\sum_{i=1}^{2}\|\nabla\omega_{n,i}\|^2
\le
E^0+\frac{C_1}{\alpha}+C_1T_*.
\]
Dividing by $\nu\lambda_I$ and taking the two suprema proves
\eqref{eq:sliding-H1}.
\end{proof}

\begin{remark}[Energy compatibility with the velocity--pressure formulation]\label{rem:primitive-variable}
At the continuous-in-space energy level, the above $L^2$ uniform-in-time estimate has a
counterpart in primitive variables. Applying
the Leray--Hopf projection $\mathcal P$ eliminates the pressure and gives the
projected velocity equation
\[
\begin{aligned}
\bm u_t+\nu A\bm u+B_{\rm vel}(\bm u,\bm u)&=\bm F,\\
A&=-\mathcal P\Delta,\\
B_{\rm vel}(\bm u,\bm v)&=\mathcal P(\bm u\cdot\nabla\bm v),
\qquad \bm F=\mathcal P\bm f.
\end{aligned}
\]
The key identities
\[
\langle B_{\rm vel}(\bm u,\bm v),\bm v\rangle=0,
\qquad
\langle A\bm v,\bm v\rangle=\|\nabla\bm v\|^2
\]
replace the corresponding vorticity identities used above. Hence the same
proof gives the same uniform $L^2$ bound and sliding-window stage dissipation
estimate for the velocity variable, with $\omega$ replaced by $\bm u$ and $f$
replaced by $\bm F$. For a fully discrete primitive-variable method, the same
conclusion requires a skew-symmetric discrete convective form and exact (or
suitably controlled) discrete incompressibility. We do not establish here
primitive-variable stage uniqueness, convergence, or a uniform $H^1$ bound.
\end{remark}

The higher-regularity estimates in the next two subsections serve two
purposes. First, their finite-time counterparts provide the regularity needed
to control the stage increments and the explicit transport term in the
convergence analysis; in particular, the numerical $H^2$ bound is used
essentially in \cref{prop:small-increments}. 
Second, eventual uniform $H^1$ and
$H^2$ bounds provide compactness and uniform regularity for the long-time
numerical dynamics. Extending the estimates proved below to eventual $H^1$
and $H^2$ bounds for the stated initial data
$\omega^0\in\dot L^s(\Omega)$, $s>2$, by means of discrete parabolic
smoothing would provide key ingredients for future work on the convergence
of global attractors and invariant measures
\cite{wang2010approximation,wang2016numerical,gottlieb2012long,tone2015double}.
Such an extension and the corresponding dynamical-systems convergence are
not claimed here.

\subsection{Uniform-in-time \texorpdfstring{$H^1$}{H1} norm estimate}

\begin{theorem}\label{lem:h1-uniform-bound}
Assume that $\omega^0\in H^2(\Omega)$ and
$f\in L^\infty(0,\infty;L^2(\Omega))$. Suppose further that
$\tau_*<\infty$ and $\rho_*<\infty$. For every fixed $T_*>\tau_*$,
there exists a constant $C>0$,
independent of $n$ and of the particular mesh within fixed bounds for
$\tau_*$ and $\rho_*$, such that, for every $n\ge0$ and $1\le i\le2$,
\begin{equation}\label{eq:H1-stage-uniform}
\begin{aligned}
\|\nabla\omega_{n,i}\|^2
+\frac{\nu\lambda_I}{4\rho_*}\tau_n\|\Delta\omega_{n,i}\|^2
\le{}&
\frac{1}{1+\alpha t_n}
\left(
\|\nabla\omega^0\|^2
+\frac{\nu\lambda_I}{4\rho_*}\tau_0\|\Delta\omega^0\|^2
\right)\\
&+
C\left(
1+\frac{1}{\alpha(T_*-\tau_*)}
\right)M_{T_*}.
\end{aligned}
\end{equation}
Here $\alpha = 3\nu\lambda_I/(4C_P+\nu\lambda_I\tau_*)$, and
$M_{T_*}$ is defined in \eqref{eq:h1-source-window} in terms of the forcing
bound and the sliding-window estimate from \cref{cor:sliding-H1}.
\end{theorem}

\begin{proof}
Proposition~\ref{prop:stage-solvability} justifies the tests below.

Step 1. [Preliminary $H^1$ estimate]
Taking the inner product of 
\eqref{eqn:isdirk2-mr-ccsav-omega} with $-2\Delta\omega_{n,i}$ gives
\[
\begin{aligned}
&\|\nabla\omega_{n,i}\|^2-\|\nabla\omega_{n,i-1}\|^2
+\|\nabla\delta_\tau\omega_{n,i}\|^2
+2\nu\tau_n\sum_{j=1}^{i}
a_{i,j}\langle\Delta\omega_{n,j},\Delta\omega_{n,i}\rangle\\
&=
2\tau_n\sum_{j=0}^{i-1}
\hat a_{i,j}\langle f_{n,j},-\Delta\omega_{n,i}\rangle
-2\tau_nG_\omega(r_{n,i})
\sum_{j=0}^{i-1}
\hat a_{i,j}\langle B_{n,j},-\Delta\omega_{n,i}\rangle .
\end{aligned}
\]
Summing over $i=1,\ldots,k$ and applying
\cref{lem:stage-coefficient-bounds} to the implicit diffusion term yields
\begin{equation}\label{eq:h1-stage-raw}
\begin{aligned}
&\|\nabla\omega_{n,k}\|^2-\|\nabla\omega^n\|^2
+\sum_{i=1}^{k}\|\nabla\delta_\tau\omega_{n,i}\|^2
+2\nu\lambda_I\tau_n\sum_{i=1}^{k}\|\Delta\omega_{n,i}\|^2\\
\le &
2\tau_n\sum_{i=1}^{k}\sum_{j=0}^{i-1}
\hat a_{i,j}\langle f_{n,j},-\Delta\omega_{n,i}\rangle +2\tau_n\sum_{i=1}^{k} |G_\omega(r_{n,i})|
\left|
\sum_{j=0}^{i-1}
\hat a_{i,j}\langle B_{n,j},-\Delta\omega_{n,i}\rangle
\right|.
\end{aligned}
\end{equation}

Step 2. [Estimate of the right-hand side]
The forcing term is controlled by the explicit coefficient bound and Young's
inequality:
\begin{equation}\label{eq:h1-force}
\begin{aligned}
2\tau_n\sum_{i=1}^{k}\sum_{j=0}^{i-1}
\hat a_{i,j}\langle f_{n,j},-\Delta\omega_{n,i}\rangle \le
\frac{\nu\lambda_I}{2}\tau_n\sum_{i=1}^{k}\|\Delta\omega_{n,i}\|^2
+\frac{2\sigma_E^2}{\nu\lambda_I}\tau_n
\sum_{j=0}^{k-1}\|f_{n,j}\|^2 .
\end{aligned}
\end{equation}
As for the nonlinear term, since $\omega$ has zero mean, the Agmon inequality
and elliptic regularity give
\[
\|\nabla^\perp(-\Delta)^{-1}\omega\|_{L^\infty}\le C\|\omega\|^\frac12\|\nabla \omega\|^\frac12,
\qquad
\|\nabla\omega\|\le
\|\omega\|^{\frac12}\|\Delta\omega\|^{\frac12}.
\]
Consequently, we have
\begin{equation}\label{eqn:h1_b_est}
\left\| B(\omega) \right\|^2 \le C \|\omega\|^2 \|\nabla \omega\| \|\Delta \omega\|.
\end{equation}

\cref{lem:l2-uniform-bound} provides uniform bounds for $\|\omega\|$ and $|r_{n,i}|$, which implies that
$|G_\omega(r_{n,i})|=|1-r_{n,i}^2|$ is also uniformly bounded. Applying the
estimate \eqref{eqn:h1_b_est}, and combining
it with \cref{lem:stage-coefficient-bounds} and Young's inequality, we obtain
\begin{equation}\label{eq:h1-nonlinear-bound}
\begin{aligned}
&2\tau_n\sum_{i=1}^{k} |G_\omega(r_{n,i})|
\left|
\sum_{j=0}^{i-1}
\hat a_{i,j}\langle B_{n,j},-\Delta\omega_{n,i}\rangle
\right|\\
&\le
\frac{\nu\lambda_I}{2}\tau_n\sum_{i=1}^{k}\|\Delta\omega_{n,i}\|^2
+\frac{\nu\lambda_I}{4\rho_*}\tau_n\|\Delta\omega^n\|^2
+C \tau_n\sum_{j=0}^{k-1}\|\nabla\omega_{n,j}\|^2 .
\end{aligned}
\end{equation}
Here, the constant $C$ depends on the $L^2$ estimate given in \cref{lem:l2-uniform-bound}, the parameters $\nu$, $\sigma_E$, $\lambda_I$, and the maximal step ratio $\rho_*$, and is independent of $n$ and the time steps.

Step 3. [Revised $H^1$ recursion]
Combining \eqref{eq:h1-stage-raw}--\eqref{eq:h1-nonlinear-bound},
using the boundedness of the forcing, and discarding nonnegative increment
terms, we obtain
\begin{equation}\label{eq:h1-preliminary-recursion}
\begin{aligned}
&\|\nabla\omega_{n,k}\|^2-\|\nabla\omega^n\|^2
+\nu \lambda_I\tau_n\sum_{i=1}^{k}\|\Delta\omega_{n,i}\|^2\\
&\le
 \frac{\nu\lambda_I}{4\rho_*}\tau_n\|\Delta\omega^n\|^2
+ C\tau_n\Big(C_f+ \sum_{j=0}^{k-1}
\|\nabla\omega_{n,j}\|^2\Big),
\qquad 1\le k\le2.
\end{aligned}
\end{equation}

Step 4. [Uniform-in-time bound]
Taking $k=2$ in \eqref{eq:h1-preliminary-recursion} gives
\begin{equation}\label{eq:h1-k2}
\begin{aligned}
&\|\nabla\omega^{n+1}\|^2
+\nu\lambda_I\tau_n\left(\|\Delta\omega_{n,1}\|^2+\|\Delta\omega^{n+1}\|^2\right)\\
&\le
\|\nabla\omega^n\|^2 + \frac{\nu\lambda_I}{4\rho_*}\tau_n\|\Delta\omega^n\|^2
+C\tau_n \left(C_f+\|\nabla\omega^n\|^2+\|\nabla\omega_{n,1}\|^2\right).
\end{aligned}
\end{equation}
Since $\tau_{n+1}\le\rho_*\tau_n$, we have
\[
\frac{\nu\lambda_I}{4\rho_*}\tau_{n+1}
\|\Delta\omega^{n+1}\|^2
\le
\frac{\nu\lambda_I}{4}\tau_n
\|\Delta\omega^{n+1}\|^2.
\]
Adding $\frac{\nu\lambda_I}{4\rho_*}\tau_{n+1} \|\Delta\omega^{n+1}\|^2$
to both sides of \eqref{eq:h1-k2}, absorbing this term into the
viscous dissipation, we obtain
\begin{equation}\label{eq:h1-modified-energy}
\begin{aligned}
&\|\nabla\omega^{n+1}\|^2
+\frac{\nu\lambda_I}{4\rho_*}\tau_{n+1} \|\Delta\omega^{n+1}\|^2
+\nu\lambda_I\tau_n
\left(
\|\Delta\omega_{n,1}\|^2 + \frac{3}{4}\|\Delta\omega^{n+1}\|^2
\right)\\
&\le
\|\nabla\omega^n\|^2
+\frac{\nu\lambda_I}{4\rho_*}\tau_n
 \|\Delta\omega^n\|^2
+C\tau_n
\left( C_f+\|\nabla\omega^n\|^2 +\|\nabla\omega_{n,1}\|^2 \right).
\end{aligned}
\end{equation}
By Poincar\'e's inequality and
$\tau_{n+1}\le\rho_*\tau_n\le\rho_*\tau_*$, we have
\[
\begin{aligned}
\|\nabla\omega^{n+1}\|^2 +\frac{\nu\lambda_I}{4\rho_*}\tau_{n+1} \|\Delta\omega^{n+1}\|^2
\le \left(C_P+\frac{\nu\lambda_I}{4}\tau_*\right) \|\Delta\omega^{n+1}\|^2.
\end{aligned}
\]
Hence, setting $\alpha:= 3\nu\lambda_I/(4C_P+\nu\lambda_I\tau_*)>0$, we obtain
\[
\frac{3\nu\lambda_I}{4}\tau_n
\|\Delta\omega^{n+1}\|^2
\ge
\alpha\tau_n
\left(
\|\nabla\omega^{n+1}\|^2
+\frac{\nu\lambda_I}{4\rho_*}\tau_{n+1}
 \|\Delta\omega^{n+1}\|^2
\right).
\]
Consequently,
\begin{equation}\label{eq:h1-gronwall-form}
\begin{aligned}
&(1+\alpha\tau_n)
\left(
\|\nabla\omega^{n+1}\|^2
+\frac{\nu\lambda_I}{4\rho_*}\tau_{n+1}
 \|\Delta\omega^{n+1}\|^2
\right)\\
&\le
\|\nabla\omega^n\|^2
+\frac{\nu\lambda_I}{4\rho_*}\tau_n
 \|\Delta\omega^n\|^2 +C\tau_n \left(C_f+\|\nabla\omega^n\|^2 +\|\nabla\omega_{n,1}\|^2\right).
\end{aligned}
\end{equation}

We next make the sliding-window bound for the source term explicit. For
$S>0$, set
\[
C_S^{(1)}:=
\frac{\|\omega^0\|^2+|1-r^0|^2+C_1/\alpha+C_1S}{\nu\lambda_I}.
\]
For a chosen window length $T_*>\tau_*$, define
\[
M_{T_*}:=C_fT_*+(1+\rho_*)C_{T_*+\tau_*}^{(1)}
+\tau_*\|\nabla\omega^0\|^2.
\]
By \cref{cor:sliding-H1}, the stage-$1$ contribution is controlled on the
original window. For the endpoint contribution, use
$\omega^n=\omega_{n-1,2}$ for $n\ge1$ and
$\tau_n\le\rho_*\tau_{n-1}$; this enlarges the window by at most $\tau_*$
and leaves only the displayed initial contribution when $m=0$. Thus the
source term has the
sliding-window bound
\begin{equation}\label{eq:h1-source-window}
\sup_{m\ge0}\;\sup_{N>m:\;\sum_{n=m}^{N-1}\tau_n\le T_{*}}
\sum_{n=m}^{N-1}\tau_n
\left(C_f+\|\nabla\omega^n\|^2+\|\nabla\omega_{n,1}\|^2\right)
\le M_{T_{*}}.
\end{equation}

Applying \cref{lem:discrete-gronwall} to
\eqref{eq:h1-gronwall-form}, we obtain
\begin{equation}\label{eq:h1-modified-uniform}
\begin{aligned}
\|\nabla\omega^n\|^2
+\frac{\nu\lambda_I}{4\rho_*}\tau_n
 \|\Delta\omega^n\|^2 \le &
\frac{1}{1+\alpha t_n}
\left(
\|\nabla\omega^0\|^2
+\frac{\nu\lambda_I}{4\rho_*}\tau_0
 \|\Delta\omega^0\|^2
\right)\\
&+C\left( 1+\frac{1}{\alpha(T_*-\tau_*)} \right)M_{T_*}.
\end{aligned}
\end{equation}
Taking $k=1$ in \eqref{eq:h1-preliminary-recursion} and using
\eqref{eq:h1-modified-uniform}, we further obtain
\begin{equation*}
\begin{aligned}
\|\nabla\omega_{n,1}\|^2 +\frac{\nu\lambda_I}{4\rho_*}\tau_n \|\Delta\omega_{n,1}\|^2
&\le \frac{1}{1+\alpha t_n}
\left(
\|\nabla\omega^0\|^2
+\frac{\nu\lambda_I}{4\rho_*}\tau_0
 \|\Delta\omega^0\|^2
\right)\\
&\quad
+C\left(
1+\frac{1}{\alpha(T_*-\tau_*)}
\right)M_{T_*}.
\end{aligned}
\end{equation*}
Here, the one-step source term has been absorbed into the
sliding-window bound. This completes the proof.
\end{proof}

\begin{corollary}[Uniform sliding-window \texorpdfstring{$H^2$}{H2} stage estimate]
\label{cor:sliding-H2}
Assume that $\omega^0\in H^2(\Omega)$, and
$f\in L^\infty(0,\infty;L^2(\Omega))$. 
Then, for every fixed window length $T_*>\tau_*$, 

\begin{equation}\label{eq:sliding-H2}
\sup_{m\ge0}\;
\sup_{\substack{N>m\\\sum_{n=m}^{N-1}\tau_n\le T_*}}
\sum_{n=m}^{N-1}\tau_n
\sum_{i=1}^{2}\|\Delta\omega_{n,i}\|^2
\le C_{T_*}^{(2)},
\end{equation}
where
\begin{equation}\label{eqn:slide_window_H2_bound}
C_{T_*}^{(2)}
:=
\frac{2}{\nu\lambda_I}
\left[
\|\nabla\omega^0\|^2
+\frac{\nu\lambda_I}{4\rho_*}\tau_0
 \|\Delta\omega^0\|^2
+
C\left(
1+\frac{4C_P+\nu\lambda_I\tau_*}{3\nu\lambda_I(T_*-\tau_*)}
\right)M_{T_*}
\right].
\end{equation}

Here, $M_{T_*}$ is the sliding-window source bound defined in
\eqref{eq:h1-source-window}.
\end{corollary}

\begin{proof}
Summing \eqref{eq:h1-modified-energy} from $n=m$ to $N-1$ and
discarding the remaining nonnegative terms gives
\begin{equation*}
\begin{aligned}
&\frac{\nu\lambda_I}{2}
\sum_{n=m}^{N-1}\tau_n
\left(
\|\Delta\omega_{n,1}\|^2
+\|\Delta\omega^{n+1}\|^2
\right)\\
&\le
\|\nabla\omega^m\|^2
+\frac{\nu\lambda_I}{4\rho_*}\tau_m
 \|\Delta\omega^m\|^2
+C\sum_{n=m}^{N-1}\tau_n
\left(
C_f+\|\nabla\omega^n\|^2
+\|\nabla\omega_{n,1}\|^2
\right).
\end{aligned}
\end{equation*}
By \eqref{eq:h1-modified-uniform},
\[
\begin{aligned}
\|\nabla\omega^m\|^2
+\frac{\nu\lambda_I}{4\rho_*}\tau_m
 \|\Delta\omega^m\|^2
\le&
\|\nabla\omega^0\|^2
+\frac{\nu\lambda_I}{4\rho_*}\tau_0
 \|\Delta\omega^0\|^2\\
&+
C\left(
1+\frac{1}{\alpha(T_*-\tau_*)}
\right)M_{T_*}.
\end{aligned}
\]
Moreover, \eqref{eq:h1-source-window} gives
\[
\sum_{n=m}^{N-1}\tau_n
\left(C_f+\|\nabla\omega^n\|^2 +\|\nabla\omega_{n,1}\|^2 \right) \le M_{T_*}.
\]
Combining these estimates yields
\[
\begin{aligned}
&\frac{\nu\lambda_I}{2}
\sum_{n=m}^{N-1}\tau_n
\left(
\|\Delta\omega_{n,1}\|^2
+\|\Delta\omega^{n+1}\|^2
\right)\\
&\le
\|\nabla\omega^0\|^2
+\frac{\nu\lambda_I}{4\rho_*}\tau_0
 \|\Delta\omega^0\|^2
+
C\left(
1+\frac{1}{\alpha(T_*-\tau_*)}
\right)M_{T_*}.
\end{aligned}
\]
Since $\omega_{n,2}=\omega^{n+1}$, taking the supremum over all
admissible $m$ and $N$ proves \eqref{eq:sliding-H2}.
\end{proof}

The uniform-in-time $L^2$ and $H^1$ estimates established above allow us
to close the stagewise uniqueness analysis in
\cref{prop:stage-uniqueness-small-step}. More precisely, the uniform
bounds on $\alpha_{n,i}$ and $\beta_{n,i}$ assumed there are direct consequences
of the long-time stability of the numerical stage values. Thus, these
coefficient bounds no longer need to be imposed as independent
assumptions, and uniqueness follows under a uniform time-step
restriction, as summarized below.

\begin{corollary}[Unique solvability of stage equations]
\label{cor:uniform-stage-uniqueness}
Under the assumptions of \cref{lem:l2-uniform-bound,lem:h1-uniform-bound}, fix a priori caps $\bar\tau,\bar\rho>0$ and consider meshes satisfying $\tau_*\le\bar\tau$ and $\rho_*\le\bar\rho$. Then there exist constants $M_\alpha(\bar\tau,\bar\rho) \geq 0$ and $M_\beta(\bar\tau,\bar\rho) \geq 0$, dependent on the uniform $L^2$ and $H^1$ bounds, the $L^\infty(0,\infty;L^2)$ norm of $f$, and the fixed mesh caps, but independent of $n$, $i$, and the particular mesh, such that
\[
|\alpha_{n,i}|\le M_\alpha(\bar\tau,\bar\rho),
\qquad
0\le\beta_{n,i}\le M_\beta(\bar\tau,\bar\rho).
\]
Consequently, all stage equations are uniquely solvable if
\begin{equation}\label{eq:uniform-uniqueness-condition}
\tau_*M_\alpha(\bar\tau,\bar\rho)
+\frac43\tau_*^2M_\beta(\bar\tau,\bar\rho)
<1+\gamma\eta\tau_*.
\end{equation}
\end{corollary}

\begin{proof}
Set
\[
C_\omega:=
\sup_{n\ge0}\max_{0\le j\le2}\|\omega_{n,j}\|^2,
C_\omega^{(1)}:=
\sup_{n\ge0}\max_{0\le j\le2}\|\omega_{n,j}\|^2_{H^1},
C_f:=\|f\|^2_{L^\infty(0,\infty;L^2)}.
\]
These constants are finite by \cref{lem:l2-uniform-bound,lem:h1-uniform-bound}.

In two dimensions, elliptic regularity and the embedding $H^2(\Omega)\hookrightarrow L^\infty(\Omega)$ give
\[
\|B_{n,j}\|
\le
C\left\|
\nabla^\perp(-\Delta)^{-1}\omega_{n,j}
\right\|_{H^2}
\|\omega_{n,j}\|_{H^1}
\le CC_\omega^{(1)}.
\]
Since the Runge--Kutta coefficients are fixed, this gives
$\|\widetilde B_{n,i}\|\le CC_\omega^{(1)}$.

Let $L_i:=I-\nu\tau_na_{i,i}\Delta$. The standard resolvent estimates
for $L_i^{-1}$ yield
\[
\begin{aligned}
\|\omega_{n,i}^{(1)}\|
&\le
\|\omega_{n,i-1}\|
+\sum_{j=1}^{i-1}
\frac{|a_{i,j}|}{a_{i,i}}\|\omega_{n,j}\|
+\tau_n\sum_{j=0}^{i-1}
|\hat a_{i,j}|\|f_{n,j}\|\\
&\le C\left(C_\omega^{\frac{1}{2}}+\tau_* C_f^{\frac{1}{2}}\right).
\end{aligned}
\]
Consequently,
\[
|\alpha_{n,i}|
\le
\|\widetilde B_{n,i}\|\,\|\omega_{n,i}^{(1)}\|
\le
CC_\omega^{(1)}\left(C_\omega^{\frac{1}{2}}+\tau_* C_f^{\frac{1}{2}}\right),
\]
and
\[
0\le\beta_{n,i}
=
\left\langle
\widetilde B_{n,i},L_i^{-1}\widetilde B_{n,i}
\right\rangle
\le
\|\widetilde B_{n,i}\|^2
\le C (C_\omega^{(1)})^2.
\]
Thus, one may take
\begin{equation}\label{eq:alpha-beta-uniform-constants}
M_\alpha(\bar\tau,\bar\rho)
=
CC_\omega^{(1)}\left(C_\omega^{\frac{1}{2}}+\bar\tau C_f^{\frac{1}{2}}\right),
\qquad
M_\beta(\bar\tau,\bar\rho)
=
C (C_\omega^{(1)})^2,
\end{equation}
where $C>0$ depends only on the domain $\Omega$ and the fixed
Runge--Kutta coefficients.  The displayed dependence on both mesh caps records
the corresponding dependence of the uniform $H^1$ bound
$C_\omega^{(1)}$ in \cref{lem:h1-uniform-bound}.

Since $a_{i,i}=\eta=(2-\sqrt2)/2$ for both stages and
$\tau_n\le\tau_*$, condition \eqref{eq:uniform-uniqueness-condition}
implies the stagewise criterion \eqref{eq:time_step_restriction} for every
$n$ and $i$. Unique solvability therefore follows from
\cref{prop:stage-uniqueness-small-step}.
\end{proof}

\subsection{Uniform-in-time \texorpdfstring{$H^2$}{H2} norm estimate}

\begin{theorem}\label{lem:h2-uniform-bound}
Assume that $\omega^0\in H^3(\Omega)$ and
$f\in L^\infty(0,\infty;H^1(\Omega))$. 
Then the stage values of $f$ satisfy
\[
\|\nabla f_{n,j}\|^2\le C_f^{(1)}:=\|f\|^2_{L^\infty(0,\infty;H^1(\Omega))},
\qquad n\ge0,\quad 0\le j\le 1.
\]
Suppose further that $\tau_*<\infty$ and $\rho_*<\infty$.
For any fixed $T_*>\tau_*$, there exists a constant $C>0$,
independent of $n$ and of the particular mesh within fixed bounds for
$\tau_*$ and $\rho_*$, such that, for every
$n\ge0$ and $i=1,2$,
\begin{equation}\label{eq:H2-final}
\|\Delta\omega_{n,i}\|^2 \le
\frac{1}{1+\alpha_2 t_n}\|\Delta\omega^0\|^2
+ C\left( 1+\frac{1}{\alpha_2(T_*-\tau_*)} \right)M_{T_*}^{(2)},
\end{equation}
where $\alpha_2:=\nu\lambda_I/C_P$, and $M_{T_*}^{(2)}$ is defined in
\eqref{eq:h2-source-window} in terms of the $H^1$ forcing bound and the
sliding-window estimate \eqref{eqn:slide_window_H2_bound}.
\end{theorem}

\begin{proof}
Proposition~\ref{prop:stage-solvability} justifies the tests below.
Step 1. [Preliminary $H^2$ estimate]
Testing the equation \eqref{eqn:isdirk2-mr-ccsav-omega} with
$2\Delta^2\omega_{n,i}$ and integrating by parts gives
\[
\begin{aligned}
&\|\Delta\omega_{n,i}\|^2-\|\Delta\omega_{n,i-1}\|^2
+\|\Delta\delta_\tau\omega_{n,i}\|^2
+2\nu\tau_n\sum_{j=1}^{i}a_{i,j}
\langle\nabla\Delta\omega_{n,j},\nabla\Delta\omega_{n,i}\rangle\\
&=
-2\tau_n\sum_{j=0}^{i-1}\hat a_{i,j}
\langle\nabla f_{n,j},\nabla\Delta\omega_{n,i}\rangle 
+2\tau_nG_\omega(r_{n,i})
\sum_{j=0}^{i-1}\hat a_{i,j}
\langle B_{n,j},\Delta^2\omega_{n,i}\rangle .
\end{aligned}
\]
Summing over $i=1,\ldots,k$ and using
\cref{lem:stage-coefficient-bounds} for the implicit diffusion term yields
\begin{equation}\label{eq:h2-stage-raw}
    \resizebox{.9\textwidth}{!}{$
\begin{aligned}
&\|\Delta\omega_{n,k}\|^2-\|\Delta\omega^n\|^2
+\sum_{i=1}^{k}\|\Delta\delta_\tau\omega_{n,i}\|^2
+2\nu\lambda_I\tau_n\sum_{i=1}^{k}
\|\nabla\Delta\omega_{n,i}\|^2\\
&\le
2\tau_n\left|
\sum_{i=1}^{k}\sum_{j=0}^{i-1}\hat a_{i,j}
\langle\nabla f_{n,j},\nabla\Delta\omega_{n,i}\rangle\right| +2\tau_n\sum_{i=1}^{k}|G_\omega(r_{n,i})|
\left|
\sum_{j=0}^{i-1}\hat a_{i,j}
\langle B_{n,j},\Delta^2\omega_{n,i}\rangle
\right|.
\end{aligned}$}
\end{equation}

Step 2. [Estimate of the right-hand side]
The forcing term is controlled by the explicit coefficient bound and Young's
inequality:
\begin{equation}\label{eq:h2-force}
\resizebox{.9\textwidth}{!}{$
\begin{aligned}
&2\tau_n\left|
\sum_{i=1}^{k}\sum_{j=0}^{i-1}\hat a_{i,j}
\langle\nabla f_{n,j},\nabla\Delta\omega_{n,i}\rangle\right|
\le \frac{\nu\lambda_I}{2}\tau_n\sum_{i=1}^{k}\|\nabla\Delta\omega_{n,i}\|^2
+C\tau_n\sum_{j=0}^{k-1}\|\nabla f_{n,j}\|^2.
\end{aligned}$}
\end{equation}
For the nonlinear term, we have
\[
\langle B_{n,j},\Delta^2\omega_{n,i}\rangle
=
-\langle \nabla B_{n,j},\nabla\Delta\omega_{n,i}\rangle.
\]
Since
\[
B_{n,j}=B(\omega_{n,j})
=\boldsymbol u_{n,j}\cdot\nabla\omega_{n,j},
\qquad
\boldsymbol u_{n,j}
:=\nabla^\perp(-\Delta)^{-1}\omega_{n,j},
\]
the product rule, H\"older's inequality, elliptic regularity, the
Sobolev inequalities, and the uniform $L^2$ and $H^1$ bounds established in
\cref{lem:l2-uniform-bound,lem:h1-uniform-bound} yield
\[
\begin{aligned}
\|\nabla B_{n,j}\|
&\le C(\|\nabla \boldsymbol u_{n,j}\|_{L^4}\|\nabla\omega_{n,j}\|_{L^4} +\|\boldsymbol u_{n,j}\|_{L^\infty}\|\Delta\omega_{n,j}\|)\\
&\le C(\|\omega_{n,j}\|_{H^1}\|\nabla\omega_{n,j}\|^{1/2}\|\Delta\omega_{n,j}\|^{1/2} +\|\omega_{n,j}\|_{H^1}\|\Delta\omega_{n,j}\|)\\
&\le C(\|\Delta\omega_{n,j}\|^{1/2}+\|\Delta\omega_{n,j}\|).
\end{aligned}
\]
Hence, for any $\varepsilon>0$, we have
\begin{equation}\label{eq:h2-trilinear}
\begin{aligned}
\left|\langle B_{n,j},\Delta^2\omega_{n,i}\rangle\right|
\le
\|\nabla B_{n,j}\|\,\|\nabla\Delta\omega_{n,i}\| 
\le 
\varepsilon\|\nabla\Delta\omega_{n,i}\|^2
+C_\varepsilon\left(1+\|\Delta\omega_{n,j}\|^2\right).
\end{aligned}
\end{equation}
Since the $L^2$ estimate gives a uniform bound for $G_\omega(r_{n,i})$,
\cref{eq:h2-trilinear} implies
\begin{equation}\label{eq:h2-nonlinear-bound}
\begin{aligned}
&2\tau_n\sum_{i=1}^{k}|G_\omega(r_{n,i})|
\left|
\sum_{j=0}^{i-1}\hat a_{i,j}
\langle B_{n,j},\Delta^2\omega_{n,i}\rangle
\right|\\
\le &
\frac{\nu\lambda_I}{2}\tau_n\sum_{i=1}^{k}
\|\nabla\Delta\omega_{n,i}\|^2
+C\tau_n\sum_{j=0}^{k-1} \left(1+\|\Delta\omega_{n,j}\|^2\right).
\end{aligned}
\end{equation}

Step 3. [Revised $H^2$ recursion]
Combining \eqref{eq:h2-stage-raw}--\eqref{eq:h2-nonlinear-bound}, using the
boundedness of the $H^1$ forcing, and discarding nonnegative increment terms,
we obtain
\begin{equation}\label{eq:h2-preliminary-recursion}
\begin{aligned}
&\|\Delta\omega_{n,k}\|^2-\|\Delta\omega^n\|^2
+\nu\lambda_I\tau_n\sum_{i=1}^{k}\|\nabla\Delta\omega_{n,i}\|^2\\
&\le
C\tau_n\sum_{j=0}^{k-1}
\left(1+\|\Delta\omega_{n,j} \|^2 + C_f^{(1)}\right),
\qquad 1\le k\le2 .
\end{aligned}
\end{equation}
Taking $k=2$ gives
\begin{equation}\label{eq:h2-k2}
\begin{aligned}
&\|\Delta\omega^{n+1}\|^2-\|\Delta\omega^n\|^2
+\nu\lambda_I\tau_n\left(
\|\nabla\Delta\omega_{n,1}\|^2
+\|\nabla\Delta\omega^{n+1}\|^2\right)\\
&\le
C\tau_n\left(
1+\|\Delta\omega^n\|^2+\|\Delta\omega_{n,1}\|^2\right).
\end{aligned}
\end{equation}

Step 4. [Endpoint recursion and uniform-in-time bound]
By Poincar\'e's inequality, $\|\Delta\omega^{n+1}\|^2 \leq C_P\|\nabla\Delta\omega^{n+1}\|^2$.
Thus, setting $\alpha_2:=\frac{\nu\lambda_I}{C_P}$, we deduce from \eqref{eq:h2-k2} that
\begin{equation}\label{eq:h2-gronwall-form}
\begin{aligned}
(1+\alpha_2\tau_n)
\|\Delta\omega^{n+1}\|^2
\le
\|\Delta\omega^n\|^2
+C\tau_n\left(1+\|\Delta\omega^n\|^2+\|\Delta\omega_{n,1}\|^2 + C_f^{(1)}\right).
\end{aligned}
\end{equation}

The stage-$1$ contribution is controlled by \cref{cor:sliding-H2}. For the
endpoint contribution, use $\omega^n=\omega_{n-1,2}$ and
$\tau_n\le\rho_*\tau_{n-1}$ as above. The required stage window is enlarged
by at most $\tau_*$, and the case $m=0$ leaves the initial contribution
$\tau_*\|\Delta\omega^0\|^2$. Hence, for every fixed $T_*>\tau_*$,
\begin{equation}\label{eq:h2-source-window}
\sup_{m\ge0}\;\sup_{N>m:\;\sum_{n=m}^{N-1}\tau_n\le T_*}
\sum_{n=m}^{N-1}\tau_n
\left(1 + C_f^{(1)}+\|\Delta\omega^n\|^2+\|\Delta\omega_{n,1}\|^2\right)
\le M_{T_*}^{(2)}.
\end{equation}
where
\[
M_{T_*}^{(2)}:=(1+C_f^{(1)})T_*
+(1+\rho_*)C_{T_*+\tau_*}^{(2)}
+\tau_*\|\Delta\omega^0\|^2,
\]
and $C_{T_*+\tau_*}^{(2)}$ is the constant in
\cref{cor:sliding-H2} for the enlarged window.

Applying \cref{lem:discrete-gronwall} to \eqref{eq:h2-gronwall-form} yields
\begin{equation}\label{eq:h2-modified-uniform}
\|\Delta\omega^n\|^2
\le
\frac{1}{1+\alpha_2t_n}\|\Delta\omega^0\|^2 +
C\left( 1+\frac{1}{\alpha_2(T_*-\tau_*)} \right)M_{T_*}^{(2)}.
\end{equation}
Taking $k=1$ in \eqref{eq:h2-preliminary-recursion} and discarding the
nonnegative dissipative term, we obtain
\[
\|\Delta\omega_{n,1}\|^2
\le
\|\Delta\omega^n\|^2 + C\tau_n\left(1+\|\Delta\omega^n\|^2\right).
\]
Using \eqref{eq:h2-modified-uniform}, we absorb the resulting lower-order term into the sliding-window contribution and obtain the same type of bound for
$\omega_{n,1}$. This completes the proof of \eqref{eq:H2-final}.
\end{proof}

\section{Variable-step convergence analysis}\label{sec:convergence}

Let $T>0$ be fixed. We first derive two preparatory estimates for the
numerical stage values and the auxiliary variable, and then establish the
finite-time error estimate.  The uniform numerical $H^2$ estimate proved in
the preceding section is essential to the present argument: it yields an
$O(\tau_n)$ bound for the stage increments in $L^2$, which in turn makes the
forcing in the scalar auxiliary equation $O(\tau_n^2)$.  Consequently,
$r_{n,i}=O(\tau_*)$ and the factor $G_\omega(r_{n,i})=1-r_{n,i}^2$ perturbs
the nonlinear term only by $O(\tau_*^2)$.

\subsection{Preparatory estimates}

\begin{proposition}[Uniform stage increment bound]\label{prop:small-increments}
Assume that the conclusions of
\cref{lem:l2-uniform-bound,lem:h1-uniform-bound,lem:h2-uniform-bound}
hold.  In particular, suppose that the numerical stage values are uniformly
bounded in $H^2$, the forcing is uniformly bounded in $L^2$, and the scalar
correction factors $G_\omega(r_{n,i})$ are uniformly bounded.  Then there
exists a constant $C>0$, uniform over grids satisfying $\tau_*\le\bar{\tau}, \rho_*\le\bar{\rho}$, such that
\[
\|\omega_{n,i}-\omega_{n,j}\| \le C\tau_n,
\qquad n\ge0,\quad 0\le j<i\le2.
\]
\end{proposition}

\begin{proof}
Step 1. [Increment equation]
The incremental form of the vorticity equation \eqref{eqn:mr_sav_stage} gives, for $i=1,2$,
\[
\omega_{n,i}-\omega_{n,i-1}
=
\nu\tau_n\sum_{j=1}^{i}a_{i,j}\Delta\omega_{n,j}
+\tau_n\sum_{j=0}^{i-1}\hat a_{i,j}
\left(f_{n,j}-G_\omega(r_{n,i})B_{n,j}\right).
\]
We estimate the three terms on the right separately.

Step 2. [Uniform bounds for the right-hand side]
Since the RK coefficients are fixed, the uniform $H^2$ bound in
\cref{lem:h2-uniform-bound} yields
\[
\Big\| \nu\tau_n\sum_{j=1}^{i}a_{i,j}\Delta\omega_{n,j} \Big\|
\leq 
 \nu\tau_n\sum_{j=1}^{i}|a_{i,j}|\left\|\Delta\omega_{n,j} \right\|
\le C\tau_n.
\]
The forcing assumption in \cref{lem:l2-uniform-bound} gives
\[
\Big\| \tau_n\sum_{j=0}^{i-1}\hat a_{i,j}f_{n,j} \Big\|
\le
\tau_n\sum_{j=0}^{i-1}|\hat a_{i,j}|\,\|f_{n,j}\|
\le C\tau_n.
\]
It remains to bound the explicit nonlinear term. In two dimensions, elliptic
regularity and the embedding $H^2(\Omega)\hookrightarrow L^\infty(\Omega)$ give $\|B_{n,j}\| \le C\|\omega_{n,j}\|_{H^1}^2
\le C$.
Moreover, the uniform $L^2$ estimate bounds
$G_\omega(r_{n,i})=1-r_{n,i}^2$ uniformly. Therefore
\[
\Big\|
\tau_nG_\omega(r_{n,i})
\sum_{j=0}^{i-1}\hat a_{i,j}B_{n,j} \Big\|
\le
C\tau_n\sum_{j=0}^{i-1}|\hat a_{i,j}|\,\|B_{n,j}\|
\le C\tau_n.
\]
Combining these three estimates yields
\[
\|\omega_{n,i}-\omega_{n,i-1}\|\le C\tau_n,\qquad i=1,2.
\]

Step 3. [Non-neighboring stage differences]
For $0\le j<i\le2$, the difference telescopes as
\[
\omega_{n,i}-\omega_{n,j}
=\sum_{q=j+1}^{i}\left(\omega_{n,q}-\omega_{n,q-1}\right).
\]
Since there are only two stages, the neighboring increment bound gives
\[
\|\omega_{n,i}-\omega_{n,j}\|\le C\tau_n.
\]
This proves the proposition.
\end{proof}

\begin{remark}\label{rem:h2-role-convergence}
The higher initial regularity implicit in the exact-solution assumptions of
\cref{thm:variable-step-conv} is not used directly in
\cref{prop:small-increments}.  Together with the assumed forcing regularity,
it ensures that the numerical $H^2$ estimate of
\cref{lem:h2-uniform-bound} is available.  The proof of
\cref{prop:small-increments} itself uses precisely the pointwise bound
$\sup_{n,i}\|\Delta\omega_{n,i}\|<\infty$ to control the diffusive part of
the stage increment in $L^2$.
\end{remark}

\begin{proposition}[Uniform-in-time smallness of the auxiliary variable]
\label{prop:r-small}
Assume the hypotheses of
\cref{lem:l2-uniform-bound,lem:h1-uniform-bound,lem:h2-uniform-bound}.
If $r^0=0$ and $\gamma>0$, then there exists $C>0$, uniform over grids satisfying $\tau_*\le\bar{\tau}, \rho_*\le\bar{\rho}$, such that 
\[
|r_{n,i}| \le C\tau_*, \qquad n\ge0,\quad i=1,2.
\]
 If $\gamma=0$, then for every fixed
$T>0$ the same estimate holds for $t_{n+1}\le T$, with a constant that may
depend on $T$.
\end{proposition}

\begin{proof}
Step 1. [Scalar increment equation]
The scalar incremental equation has the form
\[
r_{n,i}-r_{n,i-1}
=-\gamma\tau_n\sum_{j=1}^{i}a_{i,j}r_{n,j}
-\tau_nG_r(r_{n,i})
\sum_{j=0}^{i-1}\hat a_{i,j}
\langle B_{n,j},\omega_{n,i}\rangle .
\]
Equivalently,
\begin{equation}\label{eq:r-stage-increment}
r_{n,i}+\gamma\tau_n\sum_{j=1}^{i}a_{i,j}r_{n,j}
=r_{n,i-1}
-\tau_nG_r(r_{n,i})
\sum_{j=0}^{i-1}\hat a_{i,j}
\langle B_{n,j},\omega_{n,i}\rangle .
\end{equation}

Step 2. [Size of the nonlinear scalar forcing]
By skew-symmetry,
\[
\langle B_{n,j},\omega_{n,i}\rangle
=
\langle B_{n,j},\omega_{n,i}-\omega_{n,j}\rangle.
\]
By \cref{prop:small-increments} and the uniform bound on
$B_{n,j}$ obtained in its proof,
\[
\left|\langle B_{n,j},\omega_{n,i}\rangle\right|
\le
\|B_{n,j}\|\,\|\omega_{n,i}-\omega_{n,j}\|
\le C\tau_n .
\]
The uniform $L^2$ bound also controls the auxiliary variable scalar, hence
$G_r(r_{n,i})=1+r_{n,i}$ is uniformly bounded. Therefore
\begin{equation}\label{eq:r-forcing-small}
\Big|
\tau_nG_r(r_{n,i})
\sum_{j=0}^{i-1}\hat a_{i,j}
\langle B_{n,j},\omega_{n,i}\rangle
\Big|
\le C\tau_n^2 .
\end{equation}

Step 3. [Exact endpoint recursion]
Let
\[
F_{n,i}:=
-\tau_nG_r(r_{n,i})
\sum_{j=0}^{i-1}\hat a_{i,j}
\langle B_{n,j},\omega_{n,i}\rangle. \]
By \eqref{eq:r-forcing-small}, $|F_{n,i}|\le C\tau_n^2$. Since
$a_{1,1}=a_{2,2}=\eta$ and
$a_{2,1}=1-2\eta$, the two scalar stages satisfy
\[
(1+\gamma\eta\tau_n)r_{n,1}=r^n+F_{n,1},
\]
and
\[
(1+\gamma\eta\tau_n)r^{n+1}
=\left(1-\gamma(1-2\eta)\tau_n\right)r_{n,1}+F_{n,2}.
\]
Eliminating $r_{n,1}$ gives
\begin{equation}\label{eq:r-endpoint-damped}
(1+\gamma\eta\tau_n)r^{n+1}
=R_\gamma(\tau_n)r^n+\widetilde F_n,
\end{equation}
where
\begin{equation}
R_\gamma(\tau)
:=\frac{1-\gamma(1-2\eta)\tau}{1+\gamma\eta\tau},\quad
\widetilde F_n
= R_\gamma(\tau_n)F_{n,1} + F_{n,2}.
\end{equation}
For the SDIRK2 value $\eta=1-\sqrt2/2$ and fixed nonnegative $\gamma$, one has $|R_\gamma(\tau)|<\sqrt{2}$ for all $\gamma\tau>0$. It follows directly that $|\widetilde F_n|\le C\tau_n^2$.

Consequently,
\begin{equation}\label{eq:r-one-step-damped}
(1+\gamma\eta\tau_n)|r^{n+1}| \le |R_\gamma(\tau_n)||r^n|+C\tau_n^2 .
\end{equation}

Step 4. [Iteration]
Using $\tau_n^2\le\tau_*\tau_n$, \eqref{eq:r-one-step-damped} gives
\begin{equation*}
|r^{n+1}|
\le
\left(\prod_{\ell=0}^{n}
\frac{|R_\gamma(\tau_\ell)|}{1+\gamma\eta\tau_\ell}
\right)|r^0|
+ C\tau_* \sum_{p=0}^{n}
\frac{\tau_p}{1+\gamma\eta\tau_p}
\prod_{\ell=p+1}^{n} \frac{|R_\gamma(\tau_\ell)|}{1+\gamma\eta\tau_\ell}.
\end{equation*}
Since $r^0=0$, it remains to estimate the accumulated forcing term.
For $\gamma>0$ and $\eta=1-\sqrt{2}/2$, a direct calculation gives
\[
0\le
\frac{|R_\gamma(\tau)|}{1+\gamma\eta\tau}
<1,
\qquad
\frac{\tau}{1+\gamma\eta\tau}
\le
C\left(
1-\frac{|R_\gamma(\tau)|}{1+\gamma\eta\tau}
\right),
\qquad \tau>0.
\]
Therefore,
\begin{equation}\label{eqn:r_estimate}
\begin{aligned}
|r^{n+1}|
&\le
C\tau_*
\sum_{p=0}^{n}
\frac{\tau_p}{1+\gamma\eta\tau_p}
\prod_{\ell=p+1}^{n}
\frac{|R_\gamma(\tau_\ell)|}
{1+\gamma\eta\tau_\ell}\\
&\le
C\tau_*
\sum_{p=0}^{n}
\left(
1-\frac{|R_\gamma(\tau_p)|}
{1+\gamma\eta\tau_p}
\right)
\prod_{\ell=p+1}^{n}
\frac{|R_\gamma(\tau_\ell)|}
{1+\gamma\eta\tau_\ell}\\
&=
C\tau_*
\left(
1-
\prod_{\ell=0}^{n}
\frac{|R_\gamma(\tau_\ell)|}
{1+\gamma\eta\tau_\ell}
\right)
\le C\tau_*.
\end{aligned}
\end{equation}
Hence $|r^n|\le C\tau_*$ for all $n\ge0$. 
The first-stage relation and $|F_{n,1}|\le C\tau_n^2$ then give the same
bound for $r_{n,1}$.  When $\gamma=0$, direct iteration instead gives
\[
 |r^n|\le C\sum_{p=0}^{n-1}\tau_p^2
 \le C\tau_*t_n,
\]
and the first-stage relation gives the corresponding estimate for
$r_{n,1}$.  This completes the proof.
\end{proof}

\begin{remark}[Role of damping]\label{rem:eff_damp}
According to \cref{prop:r-small}, the auxiliary variable $r^n$ remains
uniformly $O(\tau_*)$ in time when $\gamma>0$. In contrast, when
$\gamma=0$, the damping factor in \eqref{eqn:r_estimate} reduces to one,
so that the $O(\tau_n^2)$ perturbations generated at successive time
steps accumulate without decay. Consequently,
\[
|r^n|
\le C\sum_{p=0}^{n-1}\tau_p^2
\le C\tau_*t_n.
\]
This permits linear growth in time. This behavior will be illustrated by the
numerical experiment in \cref{exm:gamma-effect}.

Moreover, since $G_\omega(r)=1-r^2$, we have $G_\omega(r_{n,i})=1+O(\tau_*^2)$
uniformly in time. Therefore, the correction factor introduces only an
$O(\tau_*^2)$ perturbation into the nonlinear term. Its accumulated
effect is $O(\tau_*^2)$ on any finite time interval and therefore does not
affect the second-order accuracy of the scheme.
\end{remark}

\subsection{Optimal second-order temporal error estimate}

\begin{theorem}[Optimal second-order variable-step convergence]
\label{thm:variable-step-conv}
Let $T>0$ be fixed.  Assume that
$f\in L^\infty(0,T;V)\cap W^{2,\infty}(0,T;H)$ and that the exact
 solution starting from $\omega(0)=\omega^0$ satisfies
\[
\omega\in W^{3,\infty}(0,T;H)
\cap W^{2,\infty}(0,T;\dot H_{\rm per}^2)
\cap W^{1,\infty}(0,T;\dot H_{\rm per}^4).
\]
Suppose moreover that
$r^0=0$ and that the family of time grids satisfies
$\tau_*\le\bar\tau$ and $\rho_*\le\bar\rho$ for arbitrary but fixed finite constants
$\bar\tau$ and $\bar\rho$.  Then there exists a constant $C>0$ such that
\[
\|\omega^{n}-\omega(t_n)\| \le Ct_{n}^{\frac{1}{2}}e^{\frac{Ct_{n}}{2}}\tau_*^2,
\qquad 0<t_n\le T.
\]
The constant $C$ is independent of $n$, $\tau_*$, and the particular
admissible time grid, but may depend on $T$, $\Omega$, $\nu$, $\gamma$,
$\bar\tau$, $\bar\rho$, the SDIRK2 coefficients, the forcing norms
$\|f\|_{L^\infty(0,T;V)}$ and
$\|f\|_{W^{2,\infty}(0,T;H)}$, and the three displayed solution norms.
\end{theorem}

\begin{remark}[Sufficient regularity hypotheses]\label{rem:conv-regularity}
The assumptions in \cref{thm:variable-step-conv} are sufficient rather than
minimal.  Only estimates on $[0,T]$ are used.
\ignore{: the proofs of the auxiliary stage and higher-norm bounds may be stopped at $T$, with constants depending
on the corresponding finite-time forcing norms.  One convenient, deliberately
nonminimal set of sufficient data conditions is
\[
\omega^0\in\dot H_{\rm per}^6(\Omega),\qquad
f\in C([0,T];\dot H_{\rm per}^4)
\cap C^1([0,T];\dot H_{\rm per}^2)
\cap C^2([0,T];H).
\]
}
This kind of regularity can be derived with standard two-dimensional periodic NSE regularity, or equivalently repeated
time differentiation followed by parabolic regularity.  See, for example,
\cite[Chapters~4 and~6]{temam1995navier} and the higher periodic regularity
results in \cite{FoiasTemam1989}.  
\end{remark}

\begin{proof}
We organize the argument into five steps.  The first two construct a comparison
stage and verify its endpoint consistency; the remaining steps derive and
iterate the error-energy inequality.

\smallskip
\noindent\emph{Step 1: Construction of the reference stage.}
Since $r(0)=0$, the continuous auxiliary variable vanishes identically, and
the exact vorticity satisfies $\partial_t\omega-\nu\Delta\omega=f-B(\omega)$.
Following the reference-function/comparison-stage strategy used by Zhang and Shu \cite{zhang2004error} and
adapted to IMEX time marching in \cite{wang2015stability}
(see also \cite{hairer1993solving} for the underlying RK order calculus), we
introduce the reference-stage value
$\widetilde\omega_{n,1}$ satisfying
\begin{equation}\label{eq:conv-modified-stage}
\left(I-\nu a_{1,1}\tau_n\Delta\right)
\widetilde\omega_{n,1}
=\omega(t_n)+\hat a_{1,0}\tau_n
\bigl(f_{n,0}-B(\omega(t_n))\bigr).
\end{equation}
This comparison value isolates the lower accuracy of the first internal stage
from the endpoint consistency error. 
To make this mechanism explicit, set
\[
h:=\tau_n,\qquad \mathcal L:=\nu\Delta,\qquad
Q(t,v):=f(t)-B(v),\qquad q(t):=Q(t,\omega(t)),
\]
and, only within steps 1--2, write
\[
w_j:=\partial_t^j\omega(t_n),\qquad
q_j:=\partial_t^jq(t_n),\qquad j=0,1,2.
\]
In particular, $w_0=\omega(t_n)$ and $q_0=q(t_n)$. The equation gives
\[
w_1=\mathcal Lw_0+q_0,
\qquad
w_2=\mathcal Lw_1+q_1.
\]
Since the bilinear advection term 
$$
\mathcal B(u,v):=
\nabla^\perp(-\Delta)^{-1}u\cdot\nabla v
$$
is continuous from $H^2\times H^2$ to $H$ and
$B(v)=\mathcal B(v,v)$, the stated assumptions imply
$q\in W^{2,\infty}(0,T;H)$.

\smallskip
\noindent\emph{Step 2: Reference-stage accuracy and endpoint local
truncation error.}
In this step, $O_X(h^m)$ denotes a remainder bounded in the norm of $X$ by
$Ch^m$, uniformly for $t_n+h\le T$ and over the admissible family of time
grids. Using \eqref{eq:conv-modified-stage}, $a_{1,1}=\hat a_{1,0}=\eta$,
and the resolvent identity, we obtain
\begin{equation}\label{eq:conv-reference-expansion}
\begin{aligned}
\widetilde\omega_{n,1}
&=w_0+\eta h(I-\eta h\mathcal L)^{-1}w_1\\
&=w_0+\eta hw_1+\eta^2h^2\mathcal Lw_1+O_H(h^3).
\end{aligned}
\end{equation}
Indeed, the remainder is
\[
\eta^3h^3\mathcal L^2(I-\eta h\mathcal L)^{-1}w_1,
\]
which is bounded in $H$ because
$w_1\in\dot H_{\rm per}^4=D(\mathcal L^2)$ and the resolvent is uniformly
bounded. The same assumptions and Taylor's formula in $H^2$ give
\[
\widetilde\omega_{n,1}-\omega(t_n+\eta h)=O_{H^2}(h^2).
\]
The local Lipschitz continuity of $B:H^2\to H$ and the temporal Taylor
expansion of $q$ therefore yield
\begin{equation}\label{eq:conv-explicit-expansion}
Q(t_n+\eta h,\widetilde\omega_{n,1})
=q_0+\eta hq_1+O_H(h^2).
\end{equation}

We next expand at the exact endpoint. Because $\mathcal L$ is unbounded on
$H$, we justify the expansion by applying Taylor's formula separately to
$\omega(t_n+h)$ in $H$ and to $\mathcal L\omega(t_n+h)$ in $H$.
The assumptions on $\omega$
give
\[
\omega(t_n+h)=w_0+hw_1+\frac{h^2}{2}w_2+O_H(h^3),
\]
\[
\mathcal L\omega(t_n+h)
=\mathcal Lw_0+h\mathcal Lw_1+O_H(h^2).
\]
Consequently,
\begin{equation}\label{eq:conv-exact-endpoint-expansion}
\begin{aligned}
(I-\eta h\mathcal L)\omega(t_n+h)
={}&w_0+h\bigl((1-\eta)\mathcal Lw_0+q_0\bigr)\\
&+h^2\left[
\left(\frac12-\eta\right)\mathcal Lw_1
+\frac12q_1\right]+O_H(h^3).
\end{aligned}
\end{equation}
On the other hand, \eqref{eq:conv-reference-expansion} and
\eqref{eq:conv-explicit-expansion} show that
\begin{align}
&\widetilde\omega_{n,1}
+h(1-2\eta)\mathcal L\widetilde\omega_{n,1}
+h(\delta-\eta)q_0
+h(1-\delta)Q(t_n+\eta h,\widetilde\omega_{n,1})
\notag\\
&\quad=w_0+h\bigl((1-\eta)\mathcal Lw_0+q_0\bigr)
+h^2\left[
\eta(1-\eta)\mathcal Lw_1
+\eta(1-\delta)q_1\right]+O_H(h^3).
\label{eq:conv-comparison-expansion}
\end{align}
For the Alexander coefficients,
$
\eta(1-\eta)=\frac12-\eta,
\eta(1-\delta)=\frac12
$,
 \eqref{eq:conv-exact-endpoint-expansion} and
\eqref{eq:conv-comparison-expansion} agree through order $h^2$.
Restoring the original notation gives
\begin{equation}\label{eq:conv-consistency-endpoint}
\begin{aligned}
\left(I-\nu a_{2,2}\tau_n\Delta\right)\omega(t_{n+1})
={}&\widetilde\omega_{n,1}
+\nu\tau_na_{2,1}\Delta\widetilde\omega_{n,1}
+\tau_n\hat a_{2,0}
\bigl(f_{n,0}-B(\omega(t_n))\bigr)\\
&+\tau_n\hat a_{2,1}
\bigl(f_{n,1}-B(\widetilde\omega_{n,1})\bigr)
+\mathcal R_n,
\end{aligned}
\end{equation}
where the comparison-stage local truncation residual satisfies
(cf. the related residual constructions in
\cite{zhang2004error,wang2015stability})
\begin{equation}\label{eq:conv-local-defects}
\|\mathcal R_n\|\le C\tau_n^3.
\end{equation}
Thus the comparison stage has an $O_{H^2}(h^2)$ internal-stage defect,
consistent with stage order one. Nevertheless, its use in the second-stage
equation produces the $O_H(h^3)$ local endpoint defect required for global
second-order convergence.

\smallskip
\noindent\emph{Step 3: Error equations and nonlinear perturbation bounds.}
Define
\[
e^n:=\omega^n-\omega(t_n),
\qquad e_{n,0}:=e^n,
\]
\[
e_{n,1}:=\omega_{n,1}-\widetilde\omega_{n,1},
\qquad
e_{n,2}:=\omega_{n,2}-\omega(t_{n+1}),
\qquad
e^{n+1}:=e_{n,2},
\]
and set
\[
B_{n,0}^*:=B(\omega(t_n)),
\qquad
B_{n,1}^*:=B(\widetilde\omega_{n,1}).
\]
Subtracting \eqref{eq:conv-modified-stage} and
\eqref{eq:conv-consistency-endpoint} from the two numerical stage equations
\eqref{eqn:mr_sav_stage}, respectively, gives
\begin{subequations}\label{eq:conv-error-stages}
\begin{align}
e_{n,1}-e_{n,0}
={}&\nu\tau_na_{1,1}\Delta e_{n,1}
-\tau_n\hat a_{1,0}
\bigl(G_\omega(r_{n,1})B_{n,0}-B_{n,0}^*\bigr),
\label{eq:conv-error-stage1}\\
e_{n,2}-e_{n,1}
={}&\nu\tau_n\sum_{j=1}^2a_{2,j}\Delta e_{n,j}
-\tau_n\sum_{j=0}^1\hat a_{2,j}
\bigl(G_\omega(r_{n,2})B_{n,j}-B_{n,j}^*\bigr)
-\mathcal R_n.
\label{eq:conv-error-stage2}
\end{align}
\end{subequations}

The uniform $H^1$ bounds for the numerical and exact solutions, together
with the $O_{H^2}(\tau_n^2)$ reference-stage estimate, place all arguments
below in a common bounded $H^1$ set. On such a set,
\begin{equation}\label{eq:conv-B-lipschitz}
\|B(u)-B(v)\|_{H^{-1}}\le C\|u-v\|.
\end{equation}
Indeed, for any $w\in\dot H_{\rm per}^1(\Omega)$, the divergence-free
property of $\nabla^\perp(-\Delta)^{-1}u$ gives
\[
\begin{aligned}
|\langle B(u)-B(v),w\rangle|
\le{}&
\|u\|_{L^4}
\|\nabla^\perp(-\Delta)^{-1}(u-v)\|_{L^4}\,\|\nabla w\|\\
&+\|u-v\|
\|\nabla^\perp(-\Delta)^{-1}v\|_{L^\infty}\|\nabla w\|\\
\le{}&C\|u-v\|\,\|w\|_{H^1},
\end{aligned}
\]
where elliptic regularity and the two-dimensional Sobolev embeddings were
used. Taking the supremum over $w$ proves
\eqref{eq:conv-B-lipschitz}.

Using the finite-time form of \cref{prop:r-small}, as noted in
\cref{rem:conv-regularity}, we have, for $t_{n+1}\le T$,
\[
|G_\omega(r_{n,i})|\le C,
\qquad
|G_\omega(r_{n,i})-1|\le C\tau_*^2,
\qquad i=1,2.
\]
The constant is independent of $n$; when $\gamma=0$, it may depend on $T$.
Consequently, for $0\le j<i\le2$,
\begin{equation}\label{eq:conv-nonlinear-stage-bound}
\begin{aligned}
&\|G_\omega(r_{n,i})B_{n,j}-B_{n,j}^*\|_{H^{-1}}\\
&\quad\le
|G_\omega(r_{n,i})|
\|B_{n,j}-B_{n,j}^*\|_{H^{-1}}
+|G_\omega(r_{n,i})-1|\,\|B_{n,j}^*\|_{H^{-1}}\\
&\quad\le C\|e_{n,j}\|+C\tau_*^2.
\end{aligned}
\end{equation}

\smallskip
\noindent\emph{Step 4: Stagewise error-energy estimates.}
We follow the energy argument used in the stability analysis. Taking the
$L^2$ inner product of the $i$th equation in
\eqref{eq:conv-error-stages} with $2e_{n,i}$ and summing over
$i=1,\ldots,k$, where $k=1$ or $2$, yields
\begin{equation}\label{eq:conv-energy-sum}
\begin{aligned}
&\|e_{n,k}\|^2-\|e^n\|^2
+\sum_{i=1}^k\|e_{n,i}-e_{n,i-1}\|^2
+2\nu\tau_n\sum_{i=1}^k\sum_{j=1}^i
a_{i,j}\langle\nabla e_{n,j},\nabla e_{n,i}\rangle\\
={}&-2\tau_n\sum_{i=1}^k\sum_{j=0}^{i-1}
\hat a_{i,j}
\left\langle
G_\omega(r_{n,i})B_{n,j}-B_{n,j}^*,e_{n,i}
\right\rangle
-\boldsymbol{1}_{\{k=2\}}
\langle\mathcal R_n,2e_{n,2}\rangle.
\end{aligned}
\end{equation}
By \eqref{eq:implicit-stage-coercivity},
\begin{equation}\label{eq:conv-implicit-coercivity}
\sum_{i=1}^k\sum_{j=1}^i
a_{i,j}\langle\nabla e_{n,j},\nabla e_{n,i}\rangle
\ge\lambda_I\sum_{i=1}^k\|\nabla e_{n,i}\|^2.
\end{equation}
Furthermore, \eqref{eq:explicit-stage-bound},
\eqref{eq:conv-nonlinear-stage-bound}, the Poincar\'e inequality, and
Young's inequality give
\begin{equation}\label{eq:conv-explicit-error-bound}
\begin{aligned}
&2\tau_n\left|
\sum_{i=1}^k\sum_{j=0}^{i-1}\hat a_{i,j}
\left\langle
G_\omega(r_{n,i})B_{n,j}-B_{n,j}^*,e_{n,i}
\right\rangle\right|\\
&\qquad\le
\nu\lambda_I\tau_n\sum_{i=1}^k\|\nabla e_{n,i}\|^2
+C\tau_n\sum_{j=0}^{k-1}\|e_{n,j}\|^2
+C\tau_n\tau_*^4.
\end{aligned}
\end{equation}
For $k=2$, \eqref{eq:conv-local-defects} also gives
\begin{equation}\label{eq:conv-residual-bound}
\begin{aligned}
2|\langle\mathcal R_n,e_{n,2}\rangle|
&\le
\frac{\nu\lambda_I}{2}\tau_n\|\nabla e_{n,2}\|^2
+\frac{2C_P}{\nu\lambda_I}\tau_n^{-1}\|\mathcal R_n\|^2\\
&\le
\frac{\nu\lambda_I}{2}\tau_n\|\nabla e_{n,2}\|^2
+C\tau_n^5.
\end{aligned}
\end{equation}

Combining \eqref{eq:conv-energy-sum}--\eqref{eq:conv-residual-bound} and
discarding the nonnegative increment terms, we obtain, first with $k=1$,
\begin{equation}\label{eq:conv-stage-estimate}
\|e_{n,1}\|^2
\le(1+C\tau_n)\|e^n\|^2+C\tau_n\tau_*^4.
\end{equation}

\smallskip
\noindent\emph{Step 5: Endpoint recursion and discrete Gronwall estimate.}
Taking $k=2$ in \eqref{eq:conv-energy-sum}, using
\eqref{eq:conv-implicit-coercivity}--\eqref{eq:conv-residual-bound}, and
discarding the nonnegative stage-increment terms gives
\begin{align}
&\|e^{n+1}\|^2
+\frac{\nu\lambda_I}{2}\tau_n
\sum_{i=1}^2\|\nabla e_{n,i}\|^2
\notag\\
&\qquad\le
(1+C\tau_n)\|e^n\|^2
+C\tau_n\|e_{n,1}\|^2
+C\tau_n\tau_*^4+C\tau_n^5.
\label{eq:conv-two-stage-before-substitution}
\end{align}
Substituting \eqref{eq:conv-stage-estimate} into the right-hand side yields
\begin{align}
&\|e^{n+1}\|^2
+\frac{\nu\lambda_I}{2}\tau_n
\sum_{i=1}^2\|\nabla e_{n,i}\|^2
\notag\\
&\qquad\le
\bigl(1+C\tau_n+C\tau_n(1+C\tau_n)\bigr)\|e^n\|^2
+C\tau_n\tau_*^4
+C\tau_n^2\tau_*^4
+C\tau_n^5.
\label{eq:conv-two-stage-after-substitution}
\end{align}
Using $\tau_n\le\tau_*$ and absorbing fixed constants into $C$, we obtain
\begin{equation}\label{eq:conv-recursion}
\|e^{n+1}\|^2
+\frac{\nu\lambda_I}{2}\tau_n
\sum_{i=1}^2\|\nabla e_{n,i}\|^2
\le
(1+C\tau_n)\|e^n\|^2
+C\tau_n\tau_*^4.
\end{equation}

Since $\omega(0)=\omega^0$, we have $e^0=0$. Iterating
\eqref{eq:conv-recursion} and using the discrete Gronwall inequality gives
\begin{equation}
\begin{aligned}
\|e^{n+1}\|^2
&\le
C\tau_*^4
\sum_{p=0}^{n}\tau_p
\prod_{\ell=p+1}^{n}(1+C\tau_\ell)\\
&\le
C\tau_*^4
\sum_{p=0}^{n}\tau_p
\exp\left(C\sum_{\ell=p+1}^{n}\tau_\ell\right)\\
&\le
Ct_{n+1}e^{Ct_{n+1}}\tau_*^4.
\end{aligned}
\end{equation}
Taking square roots and relabeling the generic constant gives
\[
\|e^{n+1}\|
\le Ct_{n+1}^{1/2}e^{Ct_{n+1}/2}\tau_*^2,
\]
which is the asserted estimate.
\end{proof}

\section{Numerical Experiments}

In this section, we report the results of several numerical experiments that illustrate
(1) the temporal accuracy, (2) the stability advantage over classical schemes, (3) the long-time stability, and (4) the practical performance of the proposed adaptive scheme, including its automatic time-step control capability.

{
All numerical experiments were implemented in NumPy and run on a system equipped with an Intel Core i5-13400F CPU and 16 GB of memory.} The computational domain was the periodic box \(\Omega=(0,2\pi)\times(0,2\pi)\). Unless otherwise specified, spatial discretization employed a Fourier spectral method with \(256\) modes in each direction.
Under spatially analytic forcing, positive-time solutions of the NSE possess Gevrey regularity \cite{FoiasTemam1989}, and their Fourier coefficients therefore decay exponentially. Indeed, for the examples that we investigate below, the enstrophy in Fourier modes with radial wave number $|\boldsymbol{k}|\geq 43$ remains below $10^{-10}$. This supports our choice of 256 modes. The temporal errors reported below are computed at this fixed spatial resolution.
{
Unless otherwise specified, the reference solutions in the following experiments are computed via the ETDRK4 scheme \cite{kassam2005fourth} with $\tau_{\rm ref}=10^{-4}$. We have verified that halving the step size to $5\times10^{-5}$ over $T=10$ alters the vorticity solution by at most $6\times10^{-10}$ in the $L^2$ norm, supporting the the adequacy of the chosen step‑size.}

To characterize each flow regime, we use the Reynolds number defined as
$\mathrm{Re}:=LU/\nu=\|\bm{u}\|/\nu$.
Here $L=2\pi$ is the side length of the periodic box and
$U:=\|\bm{u}\|/L$ is the corresponding root-mean-square velocity scale.
Enstrophy \eqref{eqn:enstrophy_def} is utilized to 
 monitor the flow dynamics and assess long-time stability.

\subsection{Convergence test}
The following accuracy test has three purposes. First, we isolate the
mr-ccSAV correction by comparing the proposed method with the underlying
SDIRK2 discretization obtained by fixing $r\equiv0$. Second, we compare it
with  second-order ETD-mr-ccSAV discretizations in accuracy and cost.
Third, we test whether the designed second-order convergence persists on
nonuniform temporal grids obtained by perturbing uniform step sequences.

\begin{example}
    In this example, we consider the periodic 2D Navier--Stokes equations with viscosity $\nu=10^{-3}$, mean-reverting parameter $\gamma=1000$, and terminal times $T=1$, $1.5$, $2$, and $4$. The external force is set to be
    \begin{equation*}
        f(x,y,t)=\cos(x).
    \end{equation*}
    The initial vorticity is prescribed by a smooth trigonometric field containing multiple spatial modes,
    \begin{equation}\label{eqn:smooth_trig_vorticity}
        \omega_0(x,y) = \sum_{k,m=1}^{10} \frac{1}{(k^2+m^2)^{3/2}}\cos(kx)\cos(my).
    \end{equation}
    The initial value of the auxiliary variable is set to $r^0=0$. A Fourier spectral method with $256$ modes in each spatial direction is used for spatial discretization. The smooth periodic data and the Gevrey regularity of the solution imply rapid decay of the Fourier coefficients; the temporal errors are reported at this fixed spatial resolution. A reference solution is generated by the ETDRK4 scheme  with the uniform time step $\tau_{\rm ref}=0.1\times 2^{-10}$.
\end{example}

To examine both the pre-asymptotic and asymptotic regimes, we first take
\[
\tau=0.1\times2^{-k},\qquad
k=1,2,2.2,2.4,2.6,2.8,3,\ldots,8.
\]
Table~\ref{tab:sdirk2-fixed-convergence} reports the vorticity $L^2$ errors at
$T=1$, $1.5$, $2$, and $4$. At $T=1$, the two schemes produce essentially
identical errors over the tested range. At the longer final times, however,
the original SDIRK2 scheme becomes unstable or severely inaccurate for the
coarsest time steps, whereas the SDIRK2-mr-ccSAV scheme remains stable for all step sizes. Once the time step enters the asymptotic regime,
the errors of the two schemes become nearly indistinguishable and both exhibit
second-order convergence. Thus, the mr-ccSAV correction improves coarse-step
robustness without altering the asymptotic order of the underlying SDIRK2
discretization.

\begin{table}
\centering
\caption{Vorticity $L^2$ errors and observed convergence rates at $T=1$, $T=2$, $T=5$, $T=7$.
The reference solution is computed by ETDRK4 with
$\tau_{\rm ref}=0.1\times2^{-10}$. The time steps are given by
$\tau=0.1\times2^{-k}$, with additional levels $k=2.2,2.4,2.6,2.8$.}
\label{tab:sdirk2-fixed-convergence}
\setlength{\tabcolsep}{4pt}
\resizebox{\textwidth}{!}{
\begin{tabular}{l cc cc cc cc cc cc cc cc}
\toprule
& \multicolumn{4}{c}{$T=1$} & \multicolumn{4}{c}{$T=2$} & \multicolumn{4}{c}{$T=5$} & \multicolumn{4}{c}{$T=7$}\\
\cmidrule(lr){2-5}\cmidrule(lr){6-9}\cmidrule(lr){10-13}\cmidrule(lr){14-17}
& \multicolumn{2}{c}{SDIRK2} & \multicolumn{2}{c}{SDIRK2-mr-ccSAV} & \multicolumn{2}{c}{SDIRK2} & \multicolumn{2}{c}{SDIRK2-mr-ccSAV} & \multicolumn{2}{c}{SDIRK2} & \multicolumn{2}{c}{SDIRK2-mr-ccSAV} & \multicolumn{2}{c}{SDIRK2} & \multicolumn{2}{c}{SDIRK2-mr-ccSAV}\\
\cmidrule(lr){2-3}\cmidrule(lr){4-5}\cmidrule(lr){6-7}\cmidrule(lr){8-9}\cmidrule(lr){10-11}\cmidrule(lr){12-13}\cmidrule(lr){14-15}\cmidrule(lr){16-17}
$k$ & Error & Rate & Error & Rate & Error & Rate & Error & Rate & Error & Rate & Error & Rate & Error & Rate & Error & Rate\\
\midrule
1 & $1.9507\times10^{0}$ & -- & $1.7752\times10^{0}$ & -- & \texttt{NaN} & -- & $3.6276\times10^{0}$ & -- & \texttt{NaN} & -- & $8.0890\times10^{0}$ & -- & \texttt{NaN} & -- & $1.0337\times10^{1}$ & -- \\
2 & $7.4386\times10^{-3}$ & \texttt{pre-asymp} & $7.4386\times10^{-3}$ & \texttt{pre-asymp} & \texttt{NaN} & -- & $3.3539\times10^{0}$ & \texttt{pre-asymp} & \texttt{NaN} & -- & $7.5121\times10^{0}$ & \texttt{pre-asymp} & \texttt{NaN} & -- & $9.4415\times10^{0}$ & \texttt{pre-asymp} \\
2.2 & $5.5564\times10^{-3}$ & 2.10 & $5.5564\times10^{-3}$ & 2.10 & \texttt{NaN} & -- & $3.2491\times10^{0}$ & \texttt{pre-asymp} & \texttt{NaN} & -- & $7.3480\times10^{0}$ & \texttt{pre-asymp} & \texttt{NaN} & -- & $9.1763\times10^{0}$ & \texttt{pre-asymp} \\
2.4 & $4.1471\times10^{-3}$ & 2.11 & $4.1471\times10^{-3}$ & 2.11 & \texttt{NaN} & -- & $3.0545\times10^{0}$ & \texttt{pre-asymp} & \texttt{NaN} & -- & $7.1721\times10^{0}$ & \texttt{pre-asymp} & \texttt{NaN} & -- & $8.8754\times10^{0}$ & \texttt{pre-asymp} \\
2.6 & $3.1264\times10^{-3}$ & 2.04 & $3.1264\times10^{-3}$ & 2.04 & $4.0806\times10^{0}$ & -- & $2.0662\times10^{0}$ & \texttt{pre-asymp} & \texttt{NaN} & -- & $6.9692\times10^{0}$ & \texttt{pre-asymp} & \texttt{NaN} & -- & $8.5605\times10^{0}$ & \texttt{pre-asymp} \\
2.8 & $2.3681\times10^{-3}$ & 2.00 & $2.3681\times10^{-3}$ & 2.00 & $1.8878\times10^{-2}$ & \texttt{pre-asymp} & $1.8878\times10^{-2}$ & \texttt{pre-asymp} & \texttt{NaN} & -- & $6.7901\times10^{0}$ & \texttt{pre-asymp} & \texttt{NaN} & -- & $8.2226\times10^{0}$ & \texttt{pre-asymp} \\
3 & $1.8154\times10^{-3}$ & 1.92 & $1.8154\times10^{-3}$ & 1.92 & $1.3509\times10^{-2}$ & \texttt{pre-asymp} & $1.3509\times10^{-2}$ & \texttt{pre-asymp} & \texttt{NaN} & -- & $6.6313\times10^{0}$ & \texttt{pre-asymp} & \texttt{NaN} & -- & $7.8413\times10^{0}$ & \texttt{pre-asymp} \\
4 & $4.5242\times10^{-4}$ & 2.00 & $4.5242\times10^{-4}$ & 2.00 & $3.2193\times10^{-3}$ & 2.07 & $3.2193\times10^{-3}$ & 2.07 & \texttt{NaN} & -- & $4.7995\times10^{0}$ & \texttt{pre-asymp} & \texttt{NaN} & -- & $5.3270\times10^{0}$ & \texttt{pre-asymp} \\
5 & $1.1307\times10^{-4}$ & 2.00 & $1.1307\times10^{-4}$ & 2.00 & $7.9958\times10^{-4}$ & 2.01 & $7.9958\times10^{-4}$ & 2.01 & $1.3907\times10^{-2}$ & -- & $1.3907\times10^{-2}$ & \texttt{pre-asymp} & \texttt{NaN} & -- & $5.1523\times10^{0}$ & \texttt{pre-asymp} \\
6 & $2.8270\times10^{-5}$ & 2.00 & $2.8270\times10^{-5}$ & 2.00 & $1.9967\times10^{-4}$ & 2.00 & $1.9967\times10^{-4}$ & 2.00 & $3.4433\times10^{-3}$ & 2.01 & $3.4431\times10^{-3}$ & 2.01 & $4.8624\times10^{-3}$ & -- & $4.8330\times10^{-3}$ & \texttt{pre-asymp} \\
7 & $7.0680\times10^{-6}$ & 2.00 & $7.0680\times10^{-6}$ & 2.00 & $4.9904\times10^{-5}$ & 2.00 & $4.9904\times10^{-5}$ & 2.00 & $8.5853\times10^{-4}$ & 2.00 & $8.5844\times10^{-4}$ & 2.00 & $1.2123\times10^{-3}$ & 2.00 & $1.1871\times10^{-3}$ & 2.03 \\
8 & $1.7671\times10^{-6}$ & 2.00 & $1.7671\times10^{-6}$ & 2.00 & $1.2475\times10^{-5}$ & 2.00 & $1.2475\times10^{-5}$ & 2.00 & $2.1454\times10^{-4}$ & 2.00 & $2.1447\times10^{-4}$ & 2.00 & $3.0284\times10^{-4}$ & 2.00 & $2.8705\times10^{-4}$ & 2.05 \\
\bottomrule
\end{tabular}}
\end{table}

We next compare the SDIRK2-mr-ccSAV scheme with the BDF2-mr-SAV \cite{coleman2024efficient,han2025highly},
ETD-mr-ccSAV-MS2 \cite{wang2026unconditionally}, and ETD-mr-SAV-MS2-L \cite{wang2026linear} schemes. 
For this comparison, we set $T=1$ and
$\tau=0.01\times2^{-k}$, $k=0,\ldots,8$.
Table~\ref{tab:fixed-step-convergence-cpu} shows that all five methods attain
their expected second-order temporal accuracy. At each tested time step, the
SDIRK2-mr-ccSAV scheme agrees with the underlying SDIRK2 method to the
displayed digits and gives smaller vorticity errors than the BDF2-mr-SAV and
the two ETD-based schemes. Their errors are approximately $4$, $2.5$, and
$2.5$ times larger, respectively.
The CPU timings show that the mr-ccSAV correction introduces only modest
overhead relative to the underlying SDIRK2 method. Compared with the
BDF2-mr-SAV and ETD-based schemes tested here, SDIRK2-mr-ccSAV has a smaller
error constant and therefore attains comparable accuracy on a coarser
temporal grid and with less total CPU time. Thus, in the present
implementation, it provides a favorable balance between accuracy and
computational cost, while retaining the long-time robustness demonstrated
in the subsequent experiments.

\begin{table}[htbp]
\centering
\caption{Vorticity $L^2$ errors, observed convergence rates, and CPU times for the tested schemes at $T=1$ under nominal fixed time stepping, with the final step clipped to hit $T$ exactly when needed. Here $k$ is defined by $\tau=0.01\times 2^{-k}$.}
\label{tab:fixed-step-convergence-cpu}
\resizebox{\textwidth}{!}{
\begin{tabular}{cccccccccccccccc}
\toprule
$k$ & \multicolumn{3}{c}{SDIRK2} & \multicolumn{3}{c}{SDIRK2-mr-ccSAV} & \multicolumn{3}{c}{BDF2-mr-SAV} & \multicolumn{3}{c}{ETD-mr-ccSAV-MS2} & \multicolumn{3}{c}{ETD-mr-SAV-MS2-L} \\
\cmidrule(lr){2-4} \cmidrule(lr){5-7} \cmidrule(lr){8-10} \cmidrule(lr){11-13} \cmidrule(lr){14-16}
 & Error & Order & CPU (s) & Error & Order & CPU (s) & Error & Order & CPU (s) & Error & Order & CPU (s) & Error & Order & CPU (s) \\
\midrule
0 & $9.061\times 10^{-6}$ & -- & 0.66 & $9.061\times 10^{-6}$ & -- & 0.67 & $3.711\times 10^{-5}$ & -- & 0.59 & $2.319\times 10^{-5}$ & -- & 0.54 & $2.308\times 10^{-5}$ & -- & 0.91 \\
1 & $2.262\times 10^{-6}$ & 2.00 & 1.27 & $2.262\times 10^{-6}$ & 2.00 & 1.25 & $9.158\times 10^{-6}$ & 2.02 & 1.00 & $5.741\times 10^{-6}$ & 2.01 & 1.01 & $5.723\times 10^{-6}$ & 2.01 & 1.75 \\
2 & $5.654\times 10^{-7}$ & 2.00 & 2.53 & $5.654\times 10^{-7}$ & 2.00 & 2.53 & $2.291\times 10^{-6}$ & 2.00 & 1.86 & $1.435\times 10^{-6}$ & 2.00 & 1.94 & $1.431\times 10^{-6}$ & 2.00 & 3.33 \\
3 & $1.414\times 10^{-7}$ & 2.00 & 4.85 & $1.414\times 10^{-7}$ & 2.00 & 5.09 & $5.737\times 10^{-7}$ & 2.00 & 3.57 & $3.589\times 10^{-7}$ & 2.00 & 3.82 & $3.585\times 10^{-7}$ & 2.00 & 6.53 \\
4 & $3.534\times 10^{-8}$ & 2.00 & 9.79 & $3.534\times 10^{-8}$ & 2.00 & 10.08 & $1.436\times 10^{-7}$ & 2.00 & 7.54 & $8.978\times 10^{-8}$ & 2.00 & 7.47 & $8.973\times 10^{-8}$ & 2.00 & 13.23 \\
5 & $8.835\times 10^{-9}$ & 2.00 & 19.42 & $8.835\times 10^{-9}$ & 2.00 & 20.45 & $3.592\times 10^{-8}$ & 2.00 & 14.67 & $2.245\times 10^{-8}$ & 2.00 & 15.05 & $2.240\times 10^{-8}$ & 2.00 & 26.54 \\
6 & $2.209\times 10^{-9}$ & 2.00 & 39.15 & $2.209\times 10^{-9}$ & 2.00 & 40.73 & $8.983\times 10^{-9}$ & 2.00 & 29.00 & $5.614\times 10^{-9}$ & 2.00 & 30.85 & $5.550\times 10^{-9}$ & 2.01 & 52.38 \\
7 & $5.522\times 10^{-10}$ & 2.00 & 77.04 & $5.522\times 10^{-10}$ & 2.00 & 81.65 & $2.246\times 10^{-9}$ & 2.00 & 56.77 & $1.404\times 10^{-9}$ & 2.00 & 61.43 & $1.340\times 10^{-9}$ & 2.05 & 103.62 \\
8 & $1.381\times 10^{-10}$ & 2.00 & 154.79 & $1.381\times 10^{-10}$ & 2.00 & 167.21 & $5.616\times 10^{-10}$ & 2.00 & 114.01 & $3.509\times 10^{-10}$ & 2.00 & 121.64 & $3.149\times 10^{-10}$ & 2.09 & 207.03 \\
\bottomrule
\end{tabular}
}
\end{table}

Finally, we examine the variable-step case. Starting from a uniform grid, we perturb
each time step by a relative amplitude of $15\%$ and then rescale the resulting
sequence so that the final time remains $T=2$. The errors are reported as
functions of the total number of time steps $N$. As shown in
Table~\ref{tab:convergency_perturbed}, the SDIRK2-mr-ccSAV scheme retains its
expected second-order convergence under the perturbed time-step sequence.
These results confirm that the proposed mr-ccSAV correction preserves the designed convergence order for variable time steps.

\begin{table}[htbp]
\centering
\scriptsize
\caption{Vorticity $L^2$ errors and observed convergence rates for the
SDIRK2-mr-ccSAV scheme at $T=2$ under a $15\%$ perturbed variable-step sequence.
Here $N$ denotes the total number of time steps, and the observed rates are
based on the maximum step size $\tau_*$.}
\label{tab:convergency_perturbed}
\resizebox{\textwidth}{!}{%
\begin{tabular}{@{} l *{7}{c} @{}}
\toprule
$N$ & $2^{4}\times 10$ & $2^{5}\times 10$ & $2^{6}\times 10$ & $2^{7}\times 10$ & $2^{8}\times 10$ & $2^{9}\times 10$ & $2^{10}\times 10$  \\
\midrule
Error & $1.780\times 10^{-6}$ & $4.543\times 10^{-7}$ & $1.132\times 10^{-7}$ & $2.818\times 10^{-8}$ & $7.072\times 10^{-9}$ & $1.766\times 10^{-9}$ & $4.416\times 10^{-10}$  \\
Order & -- & 1.97 & 2.00 & 2.01 & 1.99 & 2.00 & 2.00  \\
\bottomrule
\end{tabular}}
\end{table}

\begin{example}[Effect of the mean-reversion parameter]
\label{exm:gamma-effect}
We examine how the mean-reversion parameter affects the auxiliary-variable
drift and the accuracy of the vorticity approximation.
We consider the periodic vorticity--streamfunction equation with viscosity $\nu=1/40$ and Kolmogorov-type forcing $f(x,y,t)=4\cos(4y)$.

To avoid a simple steady initial profile, we initialize the flow using the smooth vorticity \eqref{eqn:smooth_trig_vorticity} adopted for the accuracy test.
We set $r^0=0$, adopt a uniform time step $\tau=10^{-3}$ up to $T=20$, and take mean-reversion parameters $\gamma=0$, $50$, $500$, $1000$, and $2000$ for comparative analysis under the SDIRK2-mr-ccSAV scheme. The case $\gamma=0$ corresponds to the non-mean-reverting SAV/ZEC limit,
whereas positive values introduce damping in the auxiliary-variable equation.
\end{example}

Figure~\ref{fig:gamma-effect} shows the auxiliary-variable magnitude, the
relative $L^2$ vorticity error, and the enstrophy history for the five tested
values of $\gamma$. Table~\ref{tab:gamma-effect} reports selected-time
diagnostics for four representative values. The results clearly distinguish the non-mean-reverting and mean-reverting regimes.

When $\gamma=0$, the perturbations arising in the auxiliary-variable equation lack damping and can accumulate over time, which aligns with the bound presented in \cref{rem:eff_damp}.
The enstrophy furthermore deviates significantly from the reference trajectory. Consequently, the non-mean-reverting discretization cannot maintain accurate long-time approximations under this time step. 
In addition, all positive values of $\gamma$ examined effectively suppress the drift of the auxiliary variable. This observation agrees with the damping mechanism behind the uniform $O(\tau_*)$ estimate for the auxiliary variable.

Moreover, from \cref{fig:gamma-effect}, we observe that the improvement of the physical solution with respect to $\gamma$ is non-monotonic. This arises because once the drift of the auxiliary variable falls below the threshold relevant to the vorticity update, further increasing $\gamma$ yields only marginal additional gains in physical accuracy. Within this regime, the residual error is dominated by temporal discretization rather than by auxiliary-variable drift.

Furthermore, we observe that for $\gamma = 0$, substantial errors do not manifest immediately in the enstrophy. Consequently, consistency of enstrophy alone cannot guarantee comparable accuracy of the complete vorticity field.
Enstrophy should therefore be regarded as a complementary stability
diagnostic rather than a standalone measure of solution accuracy. Overall,
the experiment confirms that mean reversion is essential for controlling
long-time auxiliary-variable drift, while also indicating a saturation of
the accuracy benefit for sufficiently large $\gamma$.

\begin{figure}[htbp]
\centering
\IfFileExists{figure/mrSAV_gamma_comparison.pdf}{%
\includegraphics[width=0.98\linewidth]{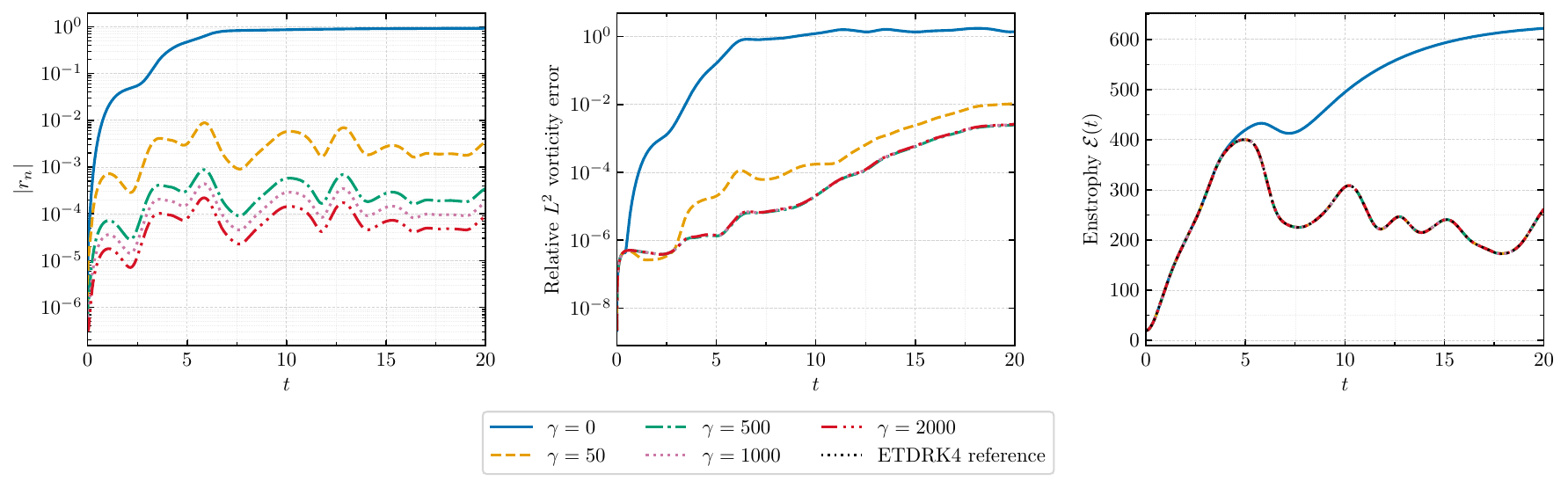}}{%
\fbox{\parbox[c][2.1in][c]{0.9\linewidth}{\centering
Figure file \texttt{mrSAV\_gamma\_comparison.pdf} was not supplied.}}}
\caption{Effect of the mean-reversion parameter for
$\gamma=0,50,500,1000,2000$. From left to right: the auxiliary-variable
magnitude $|r^n|$, the relative $L^2$ vorticity error with respect to the
ETDRK4 reference solution, and the enstrophy $\mathcal E(t)$. The black
dotted curve in the right panel denotes the ETDRK4 reference enstrophy.}
\label{fig:gamma-effect}
\end{figure}

\begin{table}[htbp]
\centering
\caption{Diagnostics for the mean-reversion parameter experiment at selected times. The rows correspond to the time $t$, and the three columns grouped under each value of $\gamma$ report $|r|$, the relative vorticity error $e_\omega$, and the relative enstrophy error $e_{\mathcal E}$, respectively. The upper and lower tabular blocks contain the results for $(\gamma=0,50)$ and $(\gamma=500,1000)$, respectively.}
\label{tab:gamma-effect}
\resizebox{\textwidth}{!}{%
\begin{tabular}{c *{6}{c}}
\toprule
& \multicolumn{3}{c}{$\gamma=0$} & \multicolumn{3}{c}{$\gamma=50$} \\
\cmidrule(lr){2-4}\cmidrule(lr){5-7}
$t$ & $|r|$ & $e_\omega$ & $e_{\mathcal E}$ & $|r|$ & $e_\omega$ & $e_{\mathcal E}$ \\
\midrule
4 & $2.887\times10^{-1}$ & $3.733\times10^{-2}$ & $7.489\times10^{-3}$ & $3.882\times10^{-3}$ & $1.228\times10^{-5}$ & $3.200\times10^{-6}$ \\
8 & $8.339\times10^{-1}$ & $8.716\times10^{-1}$ & $8.687\times10^{-1}$ & $1.117\times10^{-3}$ & $6.950\times10^{-5}$ & $7.490\times10^{-6}$ \\
16 & $9.151\times10^{-1}$ & $1.484\times10^{0}$ & $1.850\times10^{0}$ & $2.128\times10^{-3}$ & $3.889\times10^{-3}$ & $7.417\times10^{-4}$ \\
20 & $9.258\times10^{-1}$ & $1.389\times10^{0}$ & $1.376\times10^{0}$ & $3.455\times10^{-3}$ & $1.091\times10^{-2}$ & $1.532\times10^{-3}$ \\
\bottomrule
\end{tabular}%
}
\par\medskip
\resizebox{\textwidth}{!}{%
\begin{tabular}{c *{6}{c}}
\toprule
& \multicolumn{3}{c}{$\gamma=500$} & \multicolumn{3}{c}{$\gamma=1000$} \\
\cmidrule(lr){2-4}\cmidrule(lr){5-7}
$t$ & $|r|$ & $e_\omega$ & $e_{\mathcal E}$ & $|r|$ & $e_\omega$ & $e_{\mathcal E}$ \\
\midrule
4 & $3.858\times10^{-4}$ & $1.206\times10^{-6}$ & $6.120\times10^{-8}$ & $1.928\times10^{-4}$ & $1.268\times10^{-6}$ & $3.697\times10^{-8}$ \\
8 & $1.137\times10^{-4}$ & $7.454\times10^{-6}$ & $7.925\times10^{-7}$ & $5.691\times10^{-5}$ & $7.682\times10^{-6}$ & $8.562\times10^{-7}$ \\
16 & $2.106\times10^{-4}$ & $9.707\times10^{-4}$ & $1.988\times10^{-4}$ & $1.052\times10^{-4}$ & $1.007\times10^{-3}$ & $2.059\times10^{-4}$ \\
20 & $3.506\times10^{-4}$ & $2.525\times10^{-3}$ & $3.183\times10^{-4}$ & $1.754\times10^{-4}$ & $2.625\times10^{-3}$ & $3.320\times10^{-4}$ \\
\bottomrule
\end{tabular}%
}
\end{table}

\subsection{Performance of the adaptive scheme}
We next examine the accuracy and computational efficiency of the adaptive
scheme.

\begin{example}[Kolmogorov forcing: adaptive time stepping]
\label{exm:kolmogorov_1}
We consider the periodic vorticity--streamfunction problem with $f(x,y)=m\cos(my)$, $m=4$, $\nu=\frac{1}{50}$. The initial streamfunction is prescribed by the isotropic Fourier
perturbation
\begin{equation}\label{eqn:iso_per}
\psi_\varepsilon(x,y)
=
\varepsilon
\sum_{\substack{\bm{k}=(k_1,k_2)\in\mathbb Z^2\\0<|\bm{k}|\le10}}
\frac{1}{|\bm{k}|^3}(\cos(k_1x)+\sin(k_1x))(\cos(k_2y)+\sin(k_2y)),
\end{equation}
with $\varepsilon=3$. In addition, we set $r^0=0$, $\gamma=1000$ in the scheme, and the final time is set to $T=30$. 

Two time-stepping strategies are compared: (1) the uniform step sizes $\tau = 0.004$, $0.002$, $0.001$, $0.0005$ are used for the SDIRK2-mr-ccSAV scheme; and (2) the adaptive SDIRK2-mr-ccSAV scheme with the parameters in the step-controller  set to be
\begin{equation}\label{adaptive_set}
\begin{aligned}
\rho &= 0.9,\quad
\mathrm{tol}_{\omega}=5\times10^{-5},\quad
\mathrm{tol}_{r}=10^{-2},\\
\varepsilon_{\mathrm{ref}} &= 10^{-12},\quad
\tau_{\min}=10^{-5},\quad
\tau_{\max}=10^{-2}.
\end{aligned}
\end{equation}
\end{example}

Figure~\ref{fig:long_time_adap} reports the adaptive time-step histories, cumulative
step counts, CPU times, the auxiliary variable, the local $L^2$ error indicator, and the reference relative errors. 
{
To compute the reference relative vorticity errors in panel (f), the fixed-step solutions are compared with an ETDRK4 reference solution computed using the uniform time step \(\tau_{\mathrm{ref}}=10^{-4}\). For each accepted adaptive step over \([t_n,t_{n+1}]\), a local ETDRK4 reference solution is computed using \(N_n=\lceil \tau_n/\tau_{ref}\rceil\) equal substeps of size \(\delta t_n=\tau_n/N_n\le \tau_{ref}\), so that the last step lands exactly at \(t_{n+1}\).}
(i) Panel (a) indicates that the step-size controller changes the time step according to the local solution activity and the prescribed tolerances. 
(ii) Panel~(f) shows that the error of the adaptive scheme remains stably controlled throughout the computation and lies between those of the fixed-step schemes with $\tau=10^{-3}$ and $\tau=5\times10^{-4}$.
(iii) Panel (b) and (c) demonstrate that the adaptive scheme uses fewer cumulative steps and less CPU time than the fixed-step scheme with similar relative errors. 
(iv) Panel (d) and (e) show the adaptive scheme's ability to control the scalar auxiliary variable and the relative error indicator in the scheme.
We conclude that the adaptive method provides a more efficient balance between accuracy and computational effort for this test case.

\begin{figure}[htbp]
\centering
\includegraphics[width=0.9\linewidth]{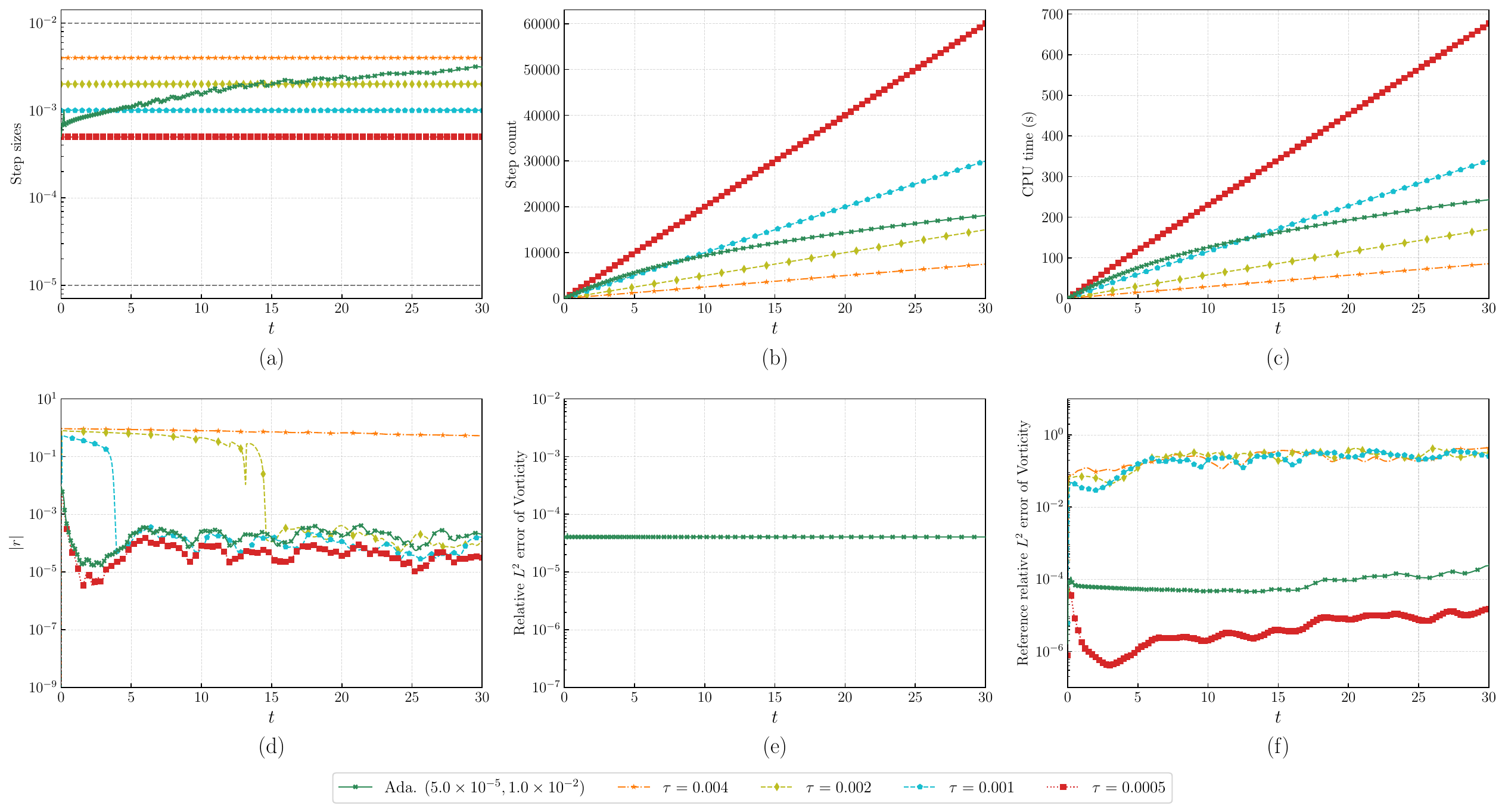}
\caption{Comparison of the fixed-step and adaptive SDIRK2-mr-ccSAV schemes
over $t\in[0,30]$. The adaptive computation uses
$(\mathrm{tol}_{\omega},\mathrm{tol}_{r})
=(5\times10^{-5},10^{-2})$ and $\rho=0.9$.
(a) Time-step histories; the horizontal dashed lines indicate the adaptive
bounds $\tau_{\min}=10^{-5}$ and $\tau_{\max}=10^{-2}$.
(b) Cumulative step counts. (c) Cumulative CPU times.
(d) Auxiliary-variable magnitude $|r|$.
(e) Relative embedded-pair discrepancy $e_{\omega,\mathrm{emb}}$.
(f) Relative vorticity error \(e_{\omega,\mathrm{ref}}\) measured against an ETDRK4 reference solution computed with time steps no larger than \(\tau_{\mathrm{ref}}=10^{-4}\).} 
\label{fig:long_time_adap}
\end{figure}

\subsection{Long-time efficiency and statistics}
The next example examines the efficiency and statistical fidelity of the adaptive strategy over a genuinely long time interval. In this regime, pointwise agreement between individual trajectories is not expected because of the sensitive dependence of the bursting dynamics on temporal perturbations. We therefore compare boundedness, representative flow scales, enstrophy distributions, and intermittent-event statistics. 

\begin{example}[Kolmogorov forcing: long time]
We consider the Kolmogorov flow $f(x,y)=m\cos(my)$ on $\Omega=(0,2\pi)^2$ with $m=4$, $\nu=1/50$, and $\gamma=1000$. The initial streamfunction is taken as the isotropic perturbation $\psi_0=\psi_\varepsilon$ in \eqref{eqn:iso_per} with $\varepsilon=4$. Both long-time simulations are computed using the SDIRK2-mr-ccSAV scheme up to $T=10000$. The fixed-step computation uses $\tau=5\times10^{-4}$, while the adaptive computation uses the controller parameters specified in \eqref{adaptive_set}. This choice of fixed step size is suggested by \cref{exm:kolmogorov_1}, as it produces smaller local relative errors in vorticity compared with the adaptive scheme for the chosen set of parameters. {
For the adaptive computation, the enstrophy trajectory is first resampled using piecewise cubic Hermite interpolation onto a uniform physical-time grid, after which the statistical quantities are evaluated.}
\end{example}

Figure~\ref{fig:adaptive_step_k} summarizes the long-time diagnostics. Panel~(a) shows that the controller reduces the step size during dynamically active episodes and increases it during relatively quiescent intervals. 
Panel~(b) shows that the fixed-step and adaptive trajectories separate over long time, as expected in this sensitive regime, while their enstrophy remains bounded. Panel~(c) demonstrates close agreement between the corresponding long-time enstrophy distributions. Panel~(d) reports the cumulative CPU time as a function of physical time. The adaptive computation requires $1{,}882{,}783$ accepted steps and $10.9$ hours, whereas the fixed-step computation requires $20{,}000{,}000$ steps and $104.6$ hours. Thus, the adaptive strategy reduces the computational cost by approximately a factor of ten while providing good approximations of the long-time statistical quantities.

\begin{figure}[htbp]
    \centering
    \IfFileExists{figure/vs_diagnostics_5e4_adaptive.pdf}{%
    \includegraphics[width=0.9\linewidth]{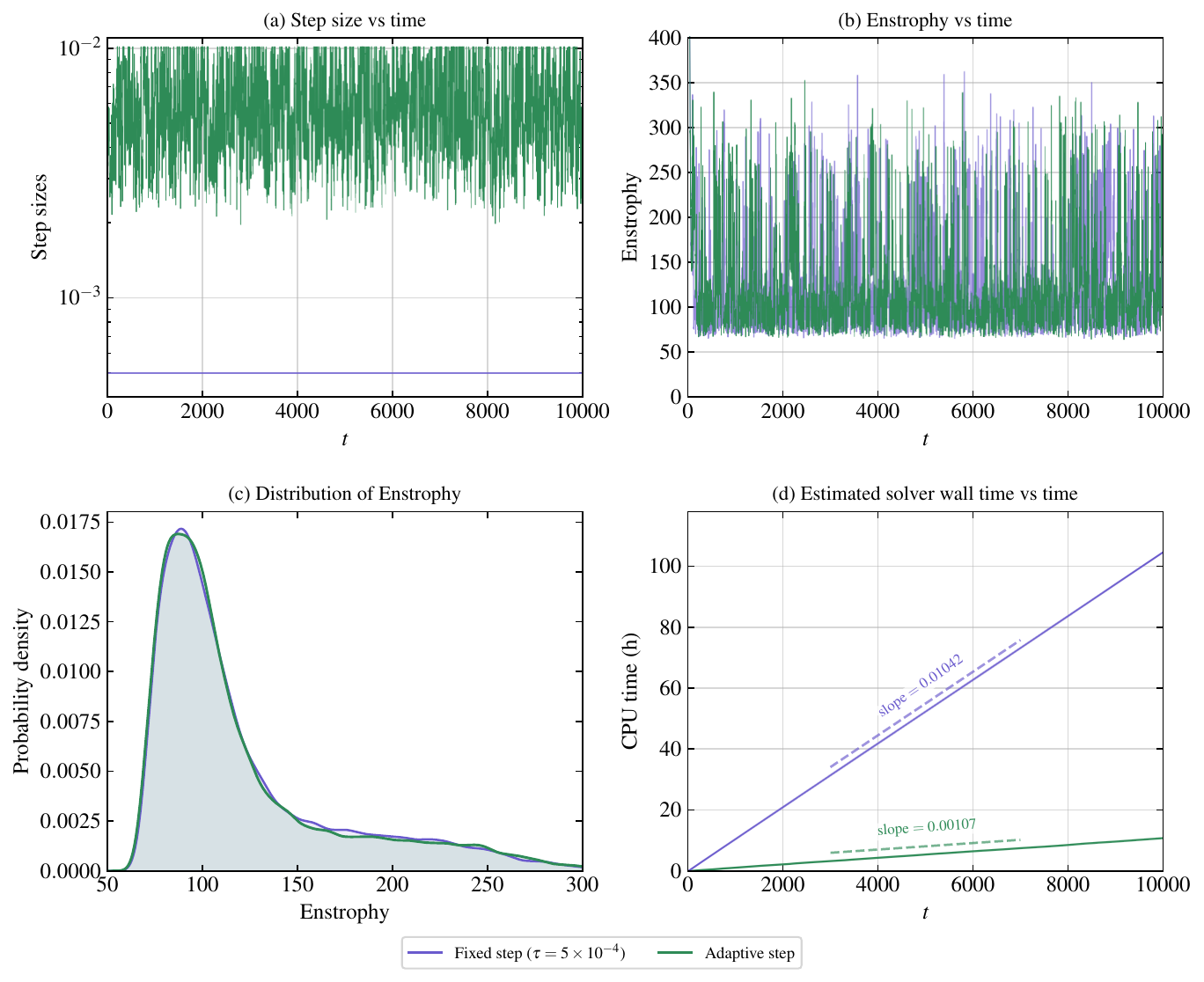}}{%
    \fbox{\parbox[c][2.1in][c]{0.85\linewidth}{\centering
    Figure file \texttt{vs\_diagnostics\_5e4\_adaptive.pdf} was not supplied.}}}
    \caption{Long-time diagnostics for the Kolmogorov flow with $\nu=1/50$ and $m=4$. (a) Adaptive step sizes $\tau_n$, with the fixed step $\tau=5\times10^{-4}$ shown for reference. (b) Enstrophy histories of the fixed-step and adaptive computations. (c) Probability density functions of enstrophy estimated on the statistically steady interval. (d) CPU time as a function of physical time for the fixed-step and adaptive-step computations.}
    \label{fig:adaptive_step_k}
\end{figure}

Notice that panel~(b) shows multiple quiescent windows and neighboring bursting events in the enstrophy evolution. We define quiescent windows as intervals during which the enstrophy remains below $\mathbb{E}(\mathcal E)+\operatorname{Std}(\mathcal E)\approx170$; their mean
duration is about $20$ eddy turnover times, so they are not short transients.
In contrast, bursts are identified when the enstrophy exceeds
$\mathbb{E}(\mathcal E)+2\operatorname{Std}(\mathcal E)\approx220$. 
The mean burst duration is approximately $3.28$ time units, or
$0.72$ eddy-turnover times.
Here the eddy-turnover time is $T_e=L/U$, where
$U:=\left(2(L^2T)^{-1}\int_0^T E(t)\,dt\right)^{1/2}$ and
$E(t):=\frac12\int_\Omega|\nabla^\perp\psi(\cdot,t)|^2$.
In particular, the mean interval between successive burst onsets is approximately
$45.78$ time units, corresponding to $10.11$ eddy-turnover times.
Consequently, an integration spanning only a few eddy-turnover times may
fall entirely within a quiescent window and miss the subsequent transition
to a bursting episode.
This observation underscores that, for
intermittent regimes, ``long time'' should be understood relative to the
recurrence time of dynamically important events, not merely relative to a
few eddy-turnover times.

\begin{table}[htbp]
\centering
\scriptsize
\caption{Long-time enstrophy statistics and time-averaged flow scales obtained by the fixed-step SDIRK2-mr-ccSAV computation with $\tau=5\times10^{-4}$ and its adaptive counterpart. The averaging interval is $1000\leq t\leq10000$, and the relative difference is measured with respect to the fixed-step result.}
\label{tab:moment_flow_scales_comparison}
\setlength{\tabcolsep}{4pt}
\begin{tabular}{@{}lccc@{}}
\toprule
Quantity & Fixed $\tau=5\times10^{-4}$ & Adaptive step & Relative difference \\
\midrule
\multicolumn{4}{@{}l}{\textit{Long-time statistics of the enstrophy}} \\
Mean $\mathbb{E}[\mathcal{E}]$ & $122.0313$ & $121.5634$ & $-0.383\%$ \\
Second moment $\mathbb{E}[\mathcal{E}^2]$ & $17375.5703$ & $17331.8726$ & $-0.251\%$ \\
Variance $\operatorname{Var}(\mathcal{E})$ & $2483.9394$ & $2554.2068$ & $2.829\%$ \\
Standard deviation $\operatorname{Std}(\mathcal{E})$ & $49.8391$ & $50.5392$ & $1.405\%$ \\
\midrule
\multicolumn{4}{@{}l}{\textit{Time-averaged flow scales}} \\
Typical velocity $U$ & $1.3238$ & $1.3261$ & $0.180\%$ \\
Representative Reynolds number $Re_U$ & $332.7061$ & $333.3063$ & $0.180\%$ \\
Eddy-turnover time $T_e$ & $4.7463$ & $4.7378$ & $-0.180\%$ \\
\bottomrule
\end{tabular}
\end{table}

We also compare the long-time enstrophy statistics in
\cref{tab:moment_flow_scales_comparison}. The mean enstrophy and the
time-averaged flow scales agree closely, while the variance and standard
deviation are more sensitive: their relative differences are approximately
$2.829\%$ and $1.405\%$, respectively. Thus the experiment supports good
agreement in the principal bulk statistics, but it does not justify a claim
of negligible bias in all fluctuation statistics.

The distributional discrepancy is small in the bulk and over the full resolved
range, but is more pronounced in the tail.

These tests do not establish convergence of the long-time statistics. They show instead that the adaptive scheme
preserves bounded long-time dynamics and reproduces the principal fixed-step
statistical diagnostics at substantially lower recorded solver cost.

\section{Conclusion}\label{sec:conclusion}

We developed an IMEX-SDIRK2-mr-ccSAV method for the forced
two-dimensional Navier--Stokes equations in periodic vorticity form. The
method treats the viscous and mean-reversion terms implicitly and evaluates
the advection term explicitly. Its stage equations reduce to two shifted
elliptic solves and a scalar algebraic equation that is either cubic or linear. We proved stagewise existence
for every positive time step, separated this result from a sufficient
small-step condition for uniqueness, and established an unconditional
variable-step uniform-in-time enstrophy bound, i.e., a $L^\infty(0,\infty; L^2)$ estimate. Under additional regularity on the initial data and finite upper-step and step-ratio bounds, we also obtained  variable-step $L^\infty(0,\infty; H^1)$ and $L^\infty(0,\infty;H^2)$
estimates. In addition, we proved the second-order temporal accuracy of the scheme on any given finite-time interval, and establish the global-in-time smallness of the scalar auxiliary variable when $r(0)=0, \gamma>0$.

The detailed solvability, higher-regularity, and convergence results in this
paper are periodic and vorticity based. The same argument is
compatible with the kinetic-energy balance in primitive variables when the
spatial discretization supplies a skew-symmetric convective form and an
appropriate divergence constraint, but primitive-variable uniqueness,
convergence, and uniform $H^1$ stability remain open. Likewise, the adaptive
experiments demonstrate bounded dynamics and useful statistical fidelity,
but do not constitute a convergence result for invariant measures.

We also expect that the method utilized in this work in deriving the unconditional uniform-in-time estimates will carry over to higher-order RK type methods with positive $S$ such as those presented in \cite{liao2026longtime}.

A natural next step is the approximation of global attractors and stationary statistical properties. The Lax-type criteria developed for dissipative numerical dynamical systems require uniform dissipativity and compactness, together with finite-time convergence uniformly for initial data drawn from the relevant discrete attracting sets\cite{wang2010approximation,wang2016numerical}. These ideas have been applied to periodic Navier--Stokes and double-diffusive discretizations in\cite{wang2012efficient,tone2015double}. For the fixed-step specialization of the present method, the uniform $L^2$ estimate and the higher-regularity bounds provide important ingredients for such an analysis. The current $H^1$ and $H^2$ estimates, however, assume correspondingly regular initial data. A complete long-time dynamical-systems analysis would first require a discrete parabolic-smoothing argument that extends these results to eventual uniform $H^1$ and $H^2$ bounds for initial data in $L^s(\Omega)$, $s>2$, or in $H^\delta(\Omega)$, $0<\delta<1$. 
The variable-step and adaptive schemes require a nonautonomous or enlarged-state dynamical-systems framework are even more challenging and are left for future study.



\bibliographystyle{amsplain}
\bibliography{references}
\end{document}